\documentclass[11pt]{article}
\usepackage{graphicx}
\usepackage{color}
\usepackage{amsmath}
\usepackage{amssymb}
\usepackage{amscd}
\usepackage{framed}
\usepackage{bbm}
\usepackage{tcolorbox}
\usepackage{fancyhdr,a4wide}
\usepackage{amsthm}
\usepackage{enumitem}

\newcommand{\R}{\mathbb{R}}
\newcommand{\inr}[1]{\left\langle #1 \right\rangle}

\newcommand{\E}{\mathbb{E}}

\newcommand{\PP}{\mathbb{P}}

\newcommand{\eps}{\varepsilon}

\newcommand{\cF}{{\cal F}}

\newtheorem{Theorem}{Theorem}[section]
\newtheorem{Lemma}[Theorem]{Lemma}

\newtheorem{Definition}[Theorem]{Definition}

\newtheorem{Corollary}[Theorem]{Corollary}
\newtheorem{Remark}[Theorem]{Remark}

\numberwithin{equation}{section}

\def \proof {\noindent {\bf Proof.}\ \ }

\def \endproof
{{\mbox{}\nolinebreak\hfill\rule{2mm}{2mm}\par\medbreak}}

\newcommand{\nrows}{m}
\newcommand{\la}{\lambda}

\newcommand{\C}{\mathbb{C}}
\newcommand{\bP}{\mathbb{P}}
\newcommand{\blue}{}
\newcommand{\red}{}
\newcommand{\green}{}

\begin{document}

\title{Fast metric embedding into the Hamming cube}

\date{}

\author{Sjoerd Dirksen\footnote{Utrecht University, Mathematical Institute (s.dirksen@uu.nl)}, Shahar Mendelson\footnote{University of Warwick, Department of Statistics and The Australian National University, Centre for Mathematics and its Applications (shahar.mendelson@gmail.com)}, Alexander Stollenwerk\footnote{UCLouvain, ICTEAM Institute (alexander.stollenwerk@uclouvain.be)} }

\maketitle

\begin{abstract}
We consider the problem of embedding a subset of $\R^n$ into a low-dimensional Hamming cube in an almost isometric way. We construct a simple, data-oblivious, and computationally efficient map that achieves this task with high probability: we first apply a specific structured random matrix, which we call the \blue{\emph{double circulant matrix}}; using that matrix requires linear storage and matrix-vector multiplication can be performed in near-linear time. We then binarize each vector by comparing each of its entries to a random threshold, selected uniformly \red{at random from} a well-chosen interval.

\blue{We estimate the} number of bits required for this encoding scheme \blue{in terms of} two natural geometric complexity parameters of the set -- its Euclidean covering numbers and its localized Gaussian complexity. \blue{The} estimate we derive turns out to be the best that one can hope for -- \blue{up to} logarithmic terms.

The key to the proof is a phenomenon of independent interest: \red{we show that} the double circulant matrix mimics the behavior of a Gaussian matrix in two important ways. \blue{First,} it maps an arbitrary set in $\R^n$ into a set of \red{well-spread} vectors. \green{\blue{Second,} it yields a fast near-isometric embedding of any finite subset of $\ell_2^n$ into $\ell_1^m$. This embedding achieves the same dimension reduction as a Gaussian matrix in near-linear time, under an optimal condition -- up to logarithmic factors -- on the number of points to be embedded. This improves a well-known construction due to Ailon and Chazelle.} 
\end{abstract}

\section{Introduction}

In modern data analysis one is frequently confronted with sets that contain a large number of points, and each point is represented by a high-dimensional vector. This high-dimensionality
causes significant storage consumption and comes at a high computational cost. In an attempt of addressing those issues, dimension reduction techniques have been used, for example, in clustering schemes \cite{MMR19}; computational geometry \cite{Indyk01}; and numerical linear algebra \cite{TYUC19,Woo14} (see, e.g., \cite{BDN15} and the references therein for many more examples). The idea is to map the given set into a lower-dimensional space, while preserving its key features. And obviously, what counts as a key feature changes according to the application one has in mind.
\par
The Gaussian random matrix $A \in \R^{m\times n}$, whose entries are independent, standard Gaussian random variables, is a surprisingly powerful and versatile tool that is frequently used in dimension reduction methods. The most basic result of that flavor is the (Gaussian formulation of the) Johnson-Lindenstrauss Lemma \cite{JL84}: if $f:\R^n\to\R^m$ is defined by $f(x)=\frac{1}{\sqrt{m}}Ax$, then for any finite set $T$ and $\epsilon>0$,
\begin{equation}
\label{eqn:J-L}
(1-\epsilon)\|x-y\|_2^2 \leq \|f(x)-f(y)\|_2^2 \leq (1+\epsilon)\|x-y\|_2^2, \qquad \text{for all } x,y\in T
\end{equation}
with high probability, provided that $m\gtrsim \epsilon^{-2}\log|T|$. Here, and throughout this article, $|T|$ denotes the number of points in $T$, $a\lesssim b$ means that $a\leq cb$ for an absolute constant $c>0$ and $a \sim b$ means $a \lesssim b$ and $b \lesssim a$.

The Johnson-Lindenstrauss Lemma is remarkable in at least two respects. First, the map $f$ is \emph{data-oblivious}, i.e., it is constructed without any prior information on the set one wishes to embed. This property is crucial in certain applications, e.g., one-pass streaming applications \cite{ClarksonW09} and data structural problems such as nearest neighbor search \cite{HarPeledIM12}. Second, the Johnson-Lindenstrauss embedding is in general \emph{optimal}: Larsen and Nelson \cite{LaN17} showed that if $\epsilon>\min\{n,|T|\}^{-0.49}$, \blue{then} any map $f:T\to \R^m$ that satisfies \eqref{eqn:J-L} must also satisfy that $m\gtrsim \epsilon^{-2}\log|T|$.

Despite this general optimality, the embedding dimension achieved by the Gaussian matrix can be substantially lower. Indeed, for a set $T$, let
\begin{equation} \label{eq:T-prime}
T^\prime = \left\{ \frac{x-y}{\|x-y\|_2} : x \not = y, \ x,y \in T \right\},
\end{equation}
denote by $G$ the standard Gaussian random vector and \blue{let}
\begin{equation} \label{eqn:GaussComplexity}
\ell_*(S) = \E \sup_{x\in S} |\langle G,x\rangle|
\end{equation}
be the \emph{Gaussian mean width} of a set $S$. A result due to Gordon \cite{Gordon88} shows that for $T \subset \R^n$, $f$ satisfies \eqref{eqn:J-L} with high probability if $m \gtrsim \epsilon^{-2}\ell_*^2(T^\prime)$. Gordon's result is an \emph{instance-optimal} version of the J-L lemma: if $T^\prime$ is as in \eqref{eq:T-prime}, then $\ell_*^2(T^\prime)$ is always upper-bounded by $c\log|T|$ for an absolute constant $c>0$, and it may be substantially lower if $T$ has a low-complexity structure, e.g., if it consists of sparse vectors or if it belongs to a low-dimensional subspace or manifold.
\par
It is also known that the Gaussian random matrix can be used to define dimension reduction schemes that go beyond the Euclidean setting. \blue{Most relevant to this work is the fact that} it is possible to embed subsets of $\ell_2^n$ into the Hamming cube $\{-1,1\}^m$ in an almost isometric way---by combining a Gaussian matrix with a straightforward binarization scheme (see \cite{OyR15,PlV14} when $T \subset S^{n-1}$ and \cite{DMS21a} for $T \subset \R^n$).
\par
At the same time, using the Gaussian matrix in \blue{dimension} reduction schemes is problematic from a computational perspective. Firstly, a typical realization of the matrix is fully populated and unstructured; thus, simply storing it requires plenty of memory ($O(mn)$). Secondly and more importantly, it takes significant time ($O(mn)$) to compute \red{a} matrix-vector product $Ax$. It is therefore highly desirable to find an alternative to the Gaussian matrix: specifically, some structured random matrix that requires less storage space and supports fast matrix-vector multiplication. Obviously, one would want that matrix to be as effective \red{in dimension reduction} as the Gaussian matrix, resulting in the best of both worlds: an optimal \blue{data-}oblivious embedding that is computationally efficient.
\par
\begin{tcolorbox}
\blue{Our main result achieves this goal for binary embeddings: we identify a computationally friendly replacement for the Gaussian matrix that leads to a near-isometric embedding of an arbitrary subset of $\R^n$ into a low-dimensional Hamming cube.}
\end{tcolorbox}
\blue{The heart of the proof of this result is \blue{to show} that the matrix we define - the double circulant matrix - mimics the behavior of the Gaussian matrix in two important ways: \red{it yields an almost isometric embedding of any subset of $\ell_2^n$ into $\ell_1^m$ and, \blue{at the same time,} it maps an arbitrary set in $\R^n$ into a set of \red{well-spread} vectors. We will make these statements} precise in Section~\ref{sec:introGaussBehav}. This behavior is remarkable because the double circulant matrix \red{has} limited randomness and its entries are \blue{strongly dependent}. Although this may be somewhat speculative, we believe that the Gaussian behavior exhibited by the double circulant matrix will have \blue{many additional applications}---well beyond the scope of binary embeddings.}

\subsection{\blue{Main result}}

\blue{Before stating our main result, we recall a binary embedding that is based on the Gaussian matrix and was studied in \cite{DMS21a}. For} a matrix $A\in \R^{m\times n}$ we consider the map $f: \R^n \to \{-1,1\}^m$ defined by
\begin{equation} \label{eqn:GaussBinEmdDefIntro}
f(x)=\operatorname{sign}(Ax+\tau),
\end{equation}
where $\tau$ is uniformly distributed in $[-\la,\la]^m$ and is independent of $A$, and the sign-function is applied component-wise. In what follows, $\ell_*(T)$ denotes the Gaussian mean width of a set $T$ (as in \eqref{eqn:GaussComplexity}) and $\mathcal{N}(T,\theta)$ is the Euclidean covering number of $T$ \red{at} scale $\theta$, i.e., the smallest number of open Euclidean balls of radius $\theta$ needed to cover the set $T$. Our starting point is the following fact, which was established in \cite{DMS21a}. Here and throughout, $d_H$ denotes the Hamming distance \red{on} $\{-1,1\}^m$.
\begin{Theorem} \label{thm:GaussianDitheredIntro}
There exist absolute constants $c_0,...,c_3$ such that the following holds. Let $T \subset \R^n$ and put $R=\sup_{t \in T} \|t\|_2$. Set $0<\delta\leq\frac{R}{2}$, $u\geq 1$ and let
$$
0<\theta \leq c_0 \frac{\delta}{\sqrt{\log(e\lambda/\delta)}}, \ \  \ \ \lambda \geq c_1 R\sqrt{\log(R/\delta)}.
$$
Suppose that
\begin{equation} \label{eq:est-m-intro}
m\geq c_2 \left( \lambda^2 \frac{\log {\cal N}(T,\theta)}{\delta^2}  +  \lambda \frac{\ell_*^2((T-T)\cap \theta B^n_2)}{\delta^3} \right).
\end{equation}
\blue{If $A\in \R^{m\times n}$ is} the standard Gaussian matrix and $\tau$ is uniformly distributed in $[-\lambda,\lambda]^m$ \blue{and independent of $A$}, then with probability at least $1-2\exp(-c_3\delta^2m/\lambda^2)$, the map $f(t)=\operatorname{sign}(At+\tau)$ satisfies 
\begin{equation} \label{eqn:GaussianDitheredText}
\sup_{x,y\in T}\left|\frac{\sqrt{2\pi}\lambda}{m}d_H(f(x),f(y))-\|x-y\|_2\right|\leq  \delta.
\end{equation}
\end{Theorem}

Although the bound on the dimension $m$ in \eqref{eq:est-m-intro} seems unnatural, it is, in fact, \emph{optimal}\blue{. We} refer the reader to \cite{DMS21a} for the proof of this surprising fact.
\par
In what follows we show that a version of Theorem \ref{thm:GaussianDitheredIntro} is true for a certain computationally friendly matrix - the double circulant matrix. To define that matrix, let $I\subset [n]$ with $|I|=m$ and set $R_I x = \sum_{i \in I} x_i e_i$. For vectors $x,y \in \R^n$, let $\Gamma_x y = x \circledast y$; thus \blue{$\Gamma_x\in \R^{n\times n}$} is the discrete convolution operator with $x$. \blue{Let $D_x = {\rm diag}(x_1,...,x_n)\in \R^{n\times n}$} be the diagonal \blue{matrix} defined by $x$, let $G\in \R^n$ be the standard Gaussian vector, and set $\eps'',\eps',\eps \in \R^n$ to be Rademacher vectors, i.e., vectors consisting of independent random variables taking values $1$ and $-1$ with equal probability. We assume throughout that $G,\eps'',\eps'$, and $\eps$ are independent.
\begin{tcolorbox}
\begin{Definition} \label{def:double-circ}
The \emph{double circulant matrix} $A\in \R^{m\times n}$ is defined by
\begin{equation} \label{eqn:doublecirculantJL}
A=\frac{1}{\sqrt{n}} R_I \Gamma_G D_{\eps''}\Gamma_{\eps'}D_{\eps}.
\end{equation}
\end{Definition}
\end{tcolorbox}
Clearly, $A$ requires $O(n)$ storage capacity, and it is well known that matrix-vector multiplication for a circulant matrix can be carried out in time $O(n\log n)$ by exploiting the fast Fourier transform.

Our main result is that the \blue{binary} embedding endowed by the double circulant matrix performs as well as the Gaussian embedding (up to logarithmic factors \red{in $n$} and with a worse success probability).

\begin{tcolorbox}
\begin{Theorem} \label{thm:DoubleCirculantEmbeddingText}
For any $\gamma\geq 1$, there exist $c_0,\ldots,c_3$ that are \blue{polynomial in $\log(n)$ and $\gamma$} such that the following holds. Fix $0<\delta<R/2$, let $T\subset RB^n_2$, and \red{set}
$$
\theta =  \frac{\delta}{\blue{c_0}\sqrt{\log(e\lambda/\delta)}}, \ \  \ \ \lambda \geq c_1 R\sqrt{\log(R/\delta)}.
$$
\blue{Suppose that $n\geq c_2\nrows$ and}
\begin{equation}
\label{eqn:DoubleCirculantEmbeddingTextNumRows}
m\geq c_3 \left( \lambda^2 \frac{\log {\cal N}(T,\theta)}{\delta^2}  +  \lambda \frac{\ell_*^2((T-T)\cap \theta B^n_2)}{\delta^3} \right).
\end{equation}
Let $A \in \R^{m \times n}$ be the double circulant matrix. If $\tau$ is uniformly distributed in $[-\lambda,\lambda]^m$ \red{and independent of $A$}, then with probability at least \blue{$1-n^{-\gamma}$}, the map $f(t)= \operatorname{sign}(At+\tau)$ satisfies 
\begin{equation*} \label{eqn:GaussianDitheredText}
\sup_{x,y\in T}\left|\frac{\sqrt{2\pi}\lambda}{m}d_H(f(x),f(y))-\|x-y\|_2\right|\leq  \delta.
\end{equation*}
\end{Theorem}
\end{tcolorbox}

\subsection{The \red{G}aussian behavior of the double circulant matrix}
\label{sec:introGaussBehav}

\blue{The} proof \red{of} Theorem~\ref{thm:GaussianDitheredIntro} is based on \red{two} well-known properties of the standard Gaussian matrix. \vskip0.3cm

First, if $A \in \R^{m \times n}$ is the Gaussian matrix, then for any $T \subset \R^n$, with high probability, $A$ is \blue{\emph{an embedding of $T$ \blue{into} $\ell_1^m$}}, in the following sense: for any $u>0$
\begin{equation} \label{eqn:ell2ell1IsomGaussian}
\bP\left(\sup_{z\in T}\Big|\frac{1}{m}\sqrt{\frac{\pi}{2}}\|Az\|_1 - \|z\|_2\Big|\leq \frac{4\ell_*(T)}{\sqrt{m}}+u\right)\leq 2e^{-mu^2/(2\mathcal{R}(T)^2)},
\end{equation}
where $\mathcal{R}(T)=\sup_{x\in T}\|x\|_2$. The proof of this fact can be found in \cite{PlV14} (see Lemma~2.1 there).
\vskip0.3cm
Second, a Gaussian matrix \emph{maps any $T$ into a set of `well-spread' vectors}: for any $\red{1\leq k\leq m}$ and $u \geq 1$, with probability at least $1-2\exp(-c u^2 k \log(em/k))$,
\begin{equation} \label{eqn:knormGaussian}
\sup_{z \in T} \|Az\|_{[k]} \leq C \left(\ell_*(T) + u \mathcal{R}(T) \sqrt{k \log(em/k)} \right),
\end{equation}
where
\begin{equation*}
\|x\|_{[k]}=\sup_{|I|=k} \left(\sum_{i\in I} x_i^2\right)^{1/2}.
\end{equation*}
The proof can be found, for example, \blue{in \cite{DMS21a} (see Theorem~2.5 there)}.\par
\blue{Together with a generic binary embedding result, stated in Theorem~\ref{thm:mainGeneric}, these two facts imply Theorem~\ref{thm:GaussianDitheredIntro}.} 
\vskip0.3cm

With that in mind, \blue{the} heart of the proof of Theorem~\ref{thm:DoubleCirculantEmbeddingText} is \blue{to show} that the double circulant matrix `acts as a Gaussian matrix'; specifically, that it satisfies suitable versions of \eqref{eqn:ell2ell1IsomGaussian} and \eqref{eqn:knormGaussian}. This behavior is surprising in view of the limited randomness \blue{in} the double circulant matrix and the \red{strong} \blue{dependence of its entries}. \blue{As a result, the approaches used to prove \eqref{eqn:ell2ell1IsomGaussian} and \eqref{eqn:knormGaussian} fail in case of the double circulant matrix. We will develop an entirely new approach to establish those properties.}\par
\blue{We first develop a general recipe for constructing a matrix that maps an arbitrary set to a collection of vectors that are well spread. The key feature that we introduce for this purpose is the notion of \emph{strong regularity}.} Intuitively, a matrix $B\in \R^{m\times n}$ is strongly regular if it acts as a Euclidean almost-isometry on sparse vectors, and also maps sparse vectors into well-spread ones (see Definition~\ref{def:regular} for a formulation of the strong regularity property). We prove that the matrix $BD_{\eps}$, obtained by randomizing the column signs of a strongly regular matrix $B$, satisfies an estimate similar to \eqref{eqn:knormGaussian} for an arbitrary set $T$ and with high probability with respect to $\eps$. The accurate formulation \blue{of this statement} can be found in Section~\ref{sec:knormEst}.
\par
Next, with the notion of strong regularity in mind, the second component of the proof of Theorem \ref{thm:DoubleCirculantEmbeddingText} is to show that
$$
B=\frac{1}{\sqrt{mn}} R_I \Gamma_G D_{\eps''}\Gamma_{\eps'}
$$
is strongly regular (obviously, $A=\sqrt{m}BD_\eps$). That \blue{fact} is established in Section~\ref{sec:reg-double-circulant} by using known (but nontrivial) tools, developed in \cite{KMR14} and \cite{DiM18b}.

Combining those facts, it follows that a typical realization of the double circulant matrix $A$ maps an arbitrary set to a collection of well-spread vectors, thus leading to an estimate as in \eqref{eqn:knormGaussian}.
\par

Once a version of \eqref{eqn:knormGaussian} is established, we turn our attention to \eqref{eqn:ell2ell1IsomGaussian}: showing that just like the Gaussian matrix, a double circulant matrix satisfies a uniform $\ell_1$\blue{-}concentration phenomenon. To that end, we first prove that for a fixed vector $y$, the random variable $\|R_I\Gamma_G y\|_1$ concentrates sharply around its mean \blue{$m\sqrt{\tfrac{2}{\pi}}\|y\|_2$} if the discrete Fourier transform of $y$ is well-spread -- see Section~\ref{sec:conv}. The exact notion of `well-spread' needed here is clarified in what follows. \blue{Next, recalling} that
$$
A=\frac{1}{\sqrt{n}} R_I \Gamma_G D_{\eps''}\Gamma_{\eps'}D_{\eps},
$$
\blue{we prove that} $\|At\|_1$ \blue{concentrates by showing that the discrete Fourier transform of} $D_{\eps''}\Gamma_{\eps'}D_{\eps}t$ is well-spread. \blue{The concentration for any fixed vector $t \in T$ is sufficiently strong to derive a uniform concentration estimate in a straightforward manner} -- see Section~\ref{sec:double-circulant-ell-1}.

\subsection{Fast $\ell_2$-$\ell_1$ dimension reduction of finite sets}

\green{As a side product, our analysis yields a result of independent interest: a new fast embedding of any finite set of points in $\ell_2^n$ into $\ell_1^m$. The embedding has runtime $O(n\log n)$ and achieves the same dimension reduction as a Gaussian matrix. This improves a well-known construction by Ailon and Chazelle \cite{AiC09}; most significantly, we remove a strong restriction on the cardinality of the set that one wishes to embed. The condition obtained here is optimal up to a poly-logarithmic factor if the goal is maximal dimension reduction -- see the next section for a detailed discussion. Note that the embedding dimension $|I|$ lies between $m/2$ and $3m/2$, say, with probability at least $1-e^{-cm}$. 
\begin{tcolorbox}
\begin{Theorem}
\label{thm:JLell1ell2Intro}
For any $\gamma\geq 1$, there exist $c_0,c_1$ that are at most polynomial in $\gamma$ such that the following holds. Consider the scaled double circulant matrix $C=\frac{1}{\nrows}\sqrt{\frac{\pi}{2}}A$, where $I$ is chosen using random selectors: $I=\{i\in [n] \ : \ \theta_i=1\}$, where $\theta_1,\ldots,\theta_n$ are independent and $1-\bP(\theta_i=0)=\bP(\theta_i=1)=m/n$. Let $T\subset \R^n$ be finite and let $\epsilon\leq \log^{-5/2}(n)$. Then, $C$ satisfies   
\begin{equation}
\label{eqn:JLell1ell2}
(1-\epsilon)\|x-y\|_2 \leq \|Cx-Cy\|_1 \leq (1+\epsilon)\|x-y\|_2, \qquad \text{for all } x,y\in T
\end{equation}
with probability at least $1-n^{-\gamma}$ if $m\geq c_0\epsilon^{-2}\log|T|$ and $|T|\leq \exp(c_1\epsilon^2 n/\log^6(n))$. 
\end{Theorem} 
\end{tcolorbox}
\begin{Remark}
\label{rem:optBitCompRS}
If $T$ is finite, then one can similarly improve Theorem~\ref{thm:DoubleCirculantEmbeddingText} by selecting $I$ using i.i.d.\ random selectors. In this case, one can show that this result remains true under the optimal condition $m\gtrsim \delta^{-2}\log|T|$ if $n\geq c\delta^{-2}\log|T|$ and $c$ is polynomial in $\log(n)$ and $\gamma$. We omit the details of this argument. 
\end{Remark}
}

\subsection{Related work}
\label{eqn:relatedWork}

\textbf{Fast $\ell_2$-$\ell_2$ dimension reduction with circulant matrices.} Numerous works have proposed and analyzed computationally efficient random matrices for $\ell_2$-$\ell_2$ dimension reduction (see, e.g., \cite{Ach03,AiC09,AiL09,BDN15,DKS10,FHM22,FrL20,HiV11,JPS20,KaN14,KrW11,Vyb11} and the references therein). We will only highlight the successful use of random circulant matrices for this task -- an approach that was first considered by Hinrichs and Vyb\'{i}ral \cite{HiV11}. They proved that the matrix $C=\frac{1}{\sqrt{m}}R_{I} \Gamma_{\eps'}D_{\eps}$ satisfies \eqref{eqn:J-L} with large probability if $m\sim \epsilon^{-2}\log^3(|T|)$. This was later improved to $m\sim \epsilon^{-2}\log^2(|T|)$ \cite{Vyb11} and shown not to be improvable further if $m|T|\leq n$ \cite{FrL20} -- however, it is possible to improve the scaling in terms of $|T|$ to $\log|T|$ at the expense of additional logarithmic factors in the dimension (by combining \cite{KMR14} and \cite{KrW11}). In fact, the main result of \cite{ORS15} (see also Theorem~\ref{thm:new} from \cite{Men21} below) shows that $C$ satisfies \eqref{eqn:J-L} with high probability under the refined condition $m \gtrsim \epsilon^{-2}\ell_*^2(T^\prime)\log^4(n)$. These results show that $C$ can serve as a computationally efficient replacement of the Gaussian matrix for $\ell_2$-$\ell_2$ dimension reduction. 
\vspace{0.2cm}
\\       
\textbf{Fast $\ell_2$-$\ell_1$ dimension reduction.} In their groundbreaking paper \cite{AiC09}, Ailon and Chazelle initiated the study of fast Johnson-Lindenstrauss transforms that use structured random matrices that support fast matrix-vector multiplication. Although most subsequent works have focused on alternative fast transforms and improved guarantees for $\ell_2$-$\ell_2$ dimension reduction (see, e.g., \cite{FHM22,HiV11,JPS20,KrW11,Vyb11}), the original work \cite{AiC09} also studied computationally efficient $\ell_2$-$\ell_1$ dimension reduction of finite point sets, motivated by approximate nearest neighbor search.\par 
Ailon and Chazelle's original transform takes the form $C=\tfrac{1}{\nrows}\sqrt{\frac{\pi}{2}}PHD_{\eps}$, where $D_{\eps}$ is as before, $H\in \R^{n\times n}$ is a normalized Hadamard matrix, and $P\in \R^{m\times n}$ is a sparsified Gaussian matrix: it has i.i.d.\ entries which are equal to $0$ with probability $1-q$ and equal to a Gaussian with mean zero and variance $1/q$ with probability $q$, where $q$ needs to be picked appropriately. It satisfies \eqref{eqn:JLell1ell2} if $q\sim\min\{\log(|T|)/(n\epsilon),1\}$ and $m\sim\epsilon^{-2}\log(|T|)$. The runtime is 
$$O(n\log(n)+\min\{n\epsilon^{-2}\log(|T|),\epsilon^{-3}\log^2(|T|)\})$$ 
so that it achieves $O(n\log n)$ runtime under an appropriate restriction on the number of points in $T$: $|T|$ must be $O_{\epsilon}(\exp(n^{1/2}))$. Essentially the same bottleneck appears in the alternative construction in \cite{AiL09}. This bottleneck was recently improved to $O_{\epsilon}(\exp(n^{1-\zeta}))$ by the method of \cite{BaK21}, but this improvement applies only for the related problem of \emph{simultaneously} embedding a subset $T'$ of $T$ in time $O(|T'| n\log n)$, where $T'$ has at least polynomial size $n^{g(\zeta)}$ (for a certain function $g$). We improve these results in two steps.\par
\green{In the case of $\ell_2$-$\ell_2$ dimension reduction, it is well known that one can construct structured random matrices with improved runtime $O(n\log n)$ (without a restriction on the number of points) if one is willing to raise the embedding dimension by a polylogarithmic factor in $|T|$ and/or $n$ (see, e.g., \cite{HiV11,JPS20,KrW11,Vyb11}). As a first step, we make an analogous contribution for fast $\ell_2$-$\ell_1$ dimension reduction: we show that the scaled double circulant matrix $C=\frac{1}{\nrows}\sqrt{\frac{\pi}{2}}A$ satisfies \eqref{eqn:JLell1ell2} with high probability if $\nrows \gtrsim \epsilon^{-2}\log(|T|)\log^{6}(n)$ (see Corollary~\ref{cor:ell-1-conc-double circulant} for $T^{\prime}$ as in \eqref{eq:T-prime})}.\par 
\green{Theorem~\ref{thm:JLell1ell2Intro} improves this further by using a double circulant matrix with $I$ picked using \emph{random selectors}: it achieves the same dimension reduction as the Gaussian matrix ($m\sim \epsilon^{-2}\log|T|$) under the significantly relaxed restriction that $|T|=O_{\epsilon}(\exp(n/\log^6(n)))$ -- which is nearly optimal if one is interested in dimension reduction (i.e., ensuring that $m\leq n$). Note that beyond the dimension reduction setting, Indyk \cite{Ind07} showed that for any distortion $\epsilon>1/\log n$ there is an embedding of \emph{all} of $\ell_2^n$ into $\ell_1^m$ with $m=O(n^{1+o(1)})$ and runtime $O(n^{1+o(1)})$.}\par  
Finally, let us mention for completeness a less closely related result on fast $\ell_2$-$\ell_1$ dimension reduction derived in the context of one-bit compressed sensing \cite{DJR20}. It was shown there that the matrix $C=\tfrac{1}{m}\sqrt{\frac{\pi}{2}} R_{I} \Gamma_G$, where $I$ is picked using i.i.d.\ random selectors, satisfies \eqref{eqn:JLell1ell2} with high probability on the set of all $s$-sparse vectors if $m\gtrsim \epsilon^{-2} s\log(n/s\epsilon)$, provided that the sparsity level $s$ is small enough (e.g., if $s\leq\epsilon\sqrt{n}$). 
\vspace{0.2cm}
\\      
\textbf{Binary encoding of Euclidean distances.} Theorem~\ref{thm:DoubleCirculantEmbeddingText} fits into a more general line of work \cite{DiS18,HuS20,InW17,InW21,Jac15,Jac17,JaCa17,JLB13,Oym16,YCP15,YBK18,ZhS20} that strives to create an \emph{efficient binary encoding} of all Euclidean distances in a given set $T$ (see also \cite{AlK17}, which focuses on inner products and \emph{squared} distances). This task consists in constructing a \blue{computationally efficient} embedding map $f:T\to\{-1,1\}^m$ and reconstruction map $d:\{-1,1\}^m\times \{-1,1\}^m\to \R$ such that for any pair $x,y\in T$, $d(f(x),f(y))$ is an accurate proxy of $\|x-y\|_2$. The binary encoding in Theorem~\ref{thm:DoubleCirculantEmbeddingText} stands out in comparison to the aforementioned works in achieving all of the following properties \emph{simultaneously}:
\begin{enumerate}[label=(\roman*)]
\item The embedding map $f$ is a metric embedding into the Hamming cube, i.e., the reconstruction map $d$ is (a constant multiple of) the natural Hamming distance, as in \cite{Oym16,YBK18}. Consequently, the number of bit operations needed to compute $d(f(x),f(y))$, is minimal - one only needs to directly compare two bit strings;
\item The embedding time $O(n\log n)$ of our construction, i.e., the time needed to compute $f(x)$ for a given $x$, is on par with the best existing results \cite{AlK17,DiS18,HuS20,Oym16,YCP15,YBK18};
\item The construction is data-oblivious, as in \cite{AlK17,DiS18,HuS20,Jac15,Jac17,JaCa17,JLB13,Oym16,YCP15,YBK18,ZhS20};
\item The bit complexity of our encoding, i.e., the number of bits (or dimension of the Hamming cube) required to encode the Euclidean distances within the given set of points, is \green{optimal for finite sets (when picking $I$ using random selectors, see Remark~\ref{rem:optBitCompRS})}. Indeed, any \emph{oblivious} random binary encoding scheme $(f,d)$ that embeds, with some given probability, any given finite set of $N$ points into $\{-1,1\}^m$ with an additive error of $\delta$, must satisfy $m\geq C \delta^{-2}\log N$ (see \cite{DMS21a}, whose proof is based on \cite{AlK17}). Similar (near-)optimal bit complexity estimates were achieved in \cite{DiS18,HuS20,Jac15,Jac17,JaCa17,YCP15,ZhS20} (see also \cite{InW21} and \cite{AlK17} for methods with optimal bit complexity for the related tasks of encoding distances up to a multiplicative error and encoding squared distances up to an additive error, respectively). 
\item The bit complexity estimate is in terms of more refined complexity measures of the dataset than the cardinality of the set, which in particular makes the result applicable to arbitrary infinite datasets in $\R^n$ (as in \cite{HuS20,Jac15,Jac17,JaCa17}).    
\end{enumerate} 
Let us comment in more detail on prior works \cite{DiS18,Oym16,YCP15,YBK18} that have the same goal of constructing a fast metric embedding into the Hamming cube. These works all considered finite sets of vectors on the unit sphere and strived to find a structured random matrix $A\in \R^{m\times n}$ that supports fast-matrix vector multiplication so that with high probability
\begin{equation}
\label{eqn:binEbSphere}
|d_H(\operatorname{sign}(Ax),\operatorname{sign}(Ay))-d_{S_{n-1}}(x,y)|\leq \delta, \qquad \text{for all } x,y\in T,
\end{equation}
where $d_{S_{n-1}}(x,y)$ is the normalized geodesic distance on the sphere. If $A$ is standard Gaussian, then it is straightforward to see that this holds under the optimal condition $m\gtrsim \delta^{-2}\log|T|$. As was remarked in \cite{DiS18,YCP15}, by simply applying a fast or sparse Johnson-Lindenstrauss transform before applying the Gaussian matrix, one obtains \eqref{eqn:binEbSphere} under the same condition with runtime $O(n\log n)$ if the number of points is small (e.g., if $\log|T|\lesssim \delta^2\sqrt{n}$ in case of a fast Johnson-Lindenstrauss transform). It was observed empirically in \cite{YBK18} that $A=R_I \Gamma_g D_{\eps}$ performs well and that performance deteriorates if $D_{\eps}$ is left out. The latter was rigorously established in \cite{DiS18}: it exhibits a two-point set for which \eqref{eqn:binEbSphere} fails with positive probability if $A=R_I \Gamma_g$. The best result in the former direction stems from \cite{Oym16}: it considers $A=R_I \Gamma_g D_{g_1}HD_{g_2}$, where $I$ consists of $m$ indices selected uniformly at random and $g,g_1,g_2$ are independent standard Gaussians, and shows that \eqref{eqn:binEbSphere} holds if $m\gtrsim \delta^{-3}\log|T|$ and $\log|T|\lesssim \delta^2 n^{1/3}$. Even though the embedding time of this transform is $O(n\log n)$ without restrictions, the guarantee is worse than for the simple combination of a standard Gaussian matrix and a fast or sparse Johnson-Lindenstrauss transform. Finally, let us mention that \cite{DiS18} (which fixed a proof gap in \cite{YCP15}) established a binary encoding with optimal bit complexity and runtime $O(n\log n)$ under the relaxed condition $\log|T|\lesssim \delta\sqrt{n}/\sqrt{\log(1/\delta)}$. However, this encoding is not a metric embedding, as it no longer uses the Hamming metric as the reconstruction map. Nevertheless, it illustrates that it is nontrivial to overcome the bottleneck on the number of points to be embedded.\par 
Our work improves over these earlier attempts to create a fast metric embedding into the Hamming cube. We do not require an artificial restriction on the number of points to be embedded, achieve embedding time $O(n\log n)$, and a bit complexity that is optimal for finite datasets. Moreover, our bit complexity estimate is in terms of more refined complexity measures than the cardinality of the dataset and in particular allows for the embedding of infinite sets. This bit complexity estimate matches (up to logarithmic factors) the one for the Gaussian embedding in Theorem~\ref{thm:GaussianDitheredIntro}, which is known to be optimal \cite{DMS21a}. These improvements are made possible by the novel Gaussian behavior of the double circulant matrix established in this work (uniform $\ell_1$-concentration and mapping into well-spread vectors), together with the use of the uniformly random shifts in \eqref{eqn:GaussBinEmdDefIntro}.\par    
Let us finally mention two binary encodings that improve over our construction at the expense of some of the listed properties (i)-(v). First, the dependence of $m$ on the additive error parameter $\delta$ can be improved (beyond the lower bound mentioned under (iv)) if an additional relative error term is present: this was achieved in \cite{HuS20} using a binary encoding that combines a fast Johnson-Lindenstrauss embedding with a so-called noise shaping method (see also \cite{ZhS20}). A downside of this encoding is that it is not a metric embedding into the Hamming cube. Second, it is possible to preserve distances up to (only) a multiplicative error. This is the goal of a different binary encoding scheme developed by Indyk and Wagner \cite{InW17,InW21}, which achieves the minimal bit complexity for a finite set for this setting. This scheme is, however, not a metric embedding, data-adaptive rather than oblivious, and has a higher computational complexity.  

\section{A generic binary embedding result}

The starting point of the proof of Theorem \ref{thm:DoubleCirculantEmbeddingText} is a generic embedding result from \cite{DMS21a}, which we now outline. Let $A\in \R^{m\times n}$ be a matrix and for a parameter $\lambda>0$, let $\tau$ be uniformly distributed in $[-\lambda,\lambda]^m$. Consider $f: \R^n \to \{-1,1\}^m$, defined by
\begin{equation} \label{eqn:GaussBinEmdDef}
 f(x)=\operatorname{sign}(Ax+\tau),
\end{equation}
where the sign-function is applied component-wise, and denote the normalized Hamming distance on $\{-1,1\}^m$ by
\begin{equation*}
\tilde{d}(x,y)= \frac{2\lambda \kappa}{m}  d_H(x,y).
\end{equation*}
The constant $\kappa$ turns out to be an absolute constant \red{in our application}, and its value is specified in what follows.

\vskip0.3cm
Let $0 < \theta < \delta$, and set $T_\theta \subset T$ to be a $\theta$-\red{net} of $T$ of minimal cardinality. Finally, assume that $A$ `acts well' on $T$ in the following sense:
\begin{description}
\item{$(a)$} \underline{\red{$A$ satisfies} uniform $\ell_1$-concentration on $T_{\theta}$:}
\begin{equation} \label{eqn:ell1ell2assump}
\sup_{x,y\in T_{\theta}}\left|\frac{\kappa}{m} \|A(x-y)\|_1 - \|x-y\|_2\right|\leq \delta.
\end{equation}
\item{$(b)$} \underline{$A$ maps $T$ to `well-spread' vectors:} For $\red{k}=\lfloor\delta m/\lambda\rfloor$ we have that
\begin{equation} \label{eqn:knormbias}
\frac{1}{\sqrt{\red{k}}} \sup_{x\in T_\theta} \|Ax\|_{[\red{k}]}\leq \lambda,
\end{equation}
and
\begin{equation} \label{eqn:knormoscillations}
\frac{1}{\sqrt{\red{k}}} \sup_{x\in (T-T)\cap \theta B_2^n} \|Ax\|_{[\red{k}]}\leq \delta.
\end{equation}
\end{description}
\vskip0.3cm
Then one has the following:
\begin{tcolorbox}
\begin{Theorem} \label{thm:mainGeneric} \cite{DMS21a}
There exist absolute constants $c_1,c_2$ and $c_3$ such that the following holds. Let
$$
m \geq c_1 \lambda^2 \kappa^2 \frac{\log\mathcal{N}(T,\theta)}{\delta^2}
$$
and assume that $A$ satisfies \eqref{eqn:ell1ell2assump}, \eqref{eqn:knormbias}, and \eqref{eqn:knormoscillations}. Then with probability at least \red{$1-2 \exp(-c_2 \delta^2 m/\lambda^2 \kappa^2)$},
\begin{equation*}
\sup_{x,y\in T}\left|\tilde{d}(f(x),f(y))-\|x-y\|_2\right|\leq c_3 (\kappa+1) \delta.
\end{equation*}
\end{Theorem}
\end{tcolorbox}
If $A$ is a standard Gaussian matrix \red{and the conditions of Theorem~\ref{thm:GaussianDitheredIntro} hold}, then \eqref{eqn:ell1ell2assump}, \eqref{eqn:knormbias}, and \eqref{eqn:knormoscillations} \blue{(with $\kappa=\sqrt{\pi/2}$)} are immediate outcomes of \eqref{eqn:ell2ell1IsomGaussian} and \eqref{eqn:knormGaussian}.

Thanks to Theorem \ref{thm:mainGeneric}, it is clear that proving the analogs of \eqref{eqn:ell1ell2assump}, \eqref{eqn:knormbias}, \red{and \eqref{eqn:knormoscillations}} would yield a binary embedding estimate---and the proof of Theorem \ref{thm:DoubleCirculantEmbeddingText}. The rest of the article is devoted to the proof of those estimates for the double circulant matrix.

\section{Strong regularity} \label{sec:knormEst}
\blue{For $x\in \R^n$,} set $\|x\|_0 =|{\rm supp}(x)|$. The vector $x$ is called $s$-sparse if $\|x\|_0 \leq s$. Denote by $U_s \subset S^{n-1}$ the set of all $s$-sparse vectors on the Euclidean unit sphere and let
$$
\Sigma_{s,n} = \{x\in \R^n \ : \ \|x\|_0\leq s, \ \|x\|_2\leq 1\}
$$
be the set of $s$-sparse vectors in the Euclidean unit ball.

A matrix $B \in \R^{m\times n}$ satisfies the $(s,\delta)$-\emph{Restricted Isometry Property} if
$$
\sup_{x \in U_s} \left| \ \|Bx\|_2^2 - \|x\|_2^2 \ \right|\leq \delta,
$$
and \red{we} denote this property by  RIP$(s,\delta)$.

\vskip0.3cm
The following definition is of crucial importance in the context of the proof.
\begin{tcolorbox}
\begin{Definition} \label{def:regular}
Let $\rho>0$ \red{and set $s_{\rho}=\lceil\rho^{-2}\rceil$.} A matrix $B\red{\in \R^{m\times n}}$ is \emph{$\red{\rho}$-regular} if it satisfies \emph{RIP}$(r,\rho\sqrt{r})$ for all $1\leq r\leq \red{s_{\rho}}$. It is \emph{$\red{\rho}$-strongly regular} if it is $\red{\rho}$-regular and, in addition, satisfies 
\begin{align*}
\sup_{x\in \Sigma_{r,n}}\|Bx\|_{[r]}\leq \rho\sqrt{r}, \quad \text{for all } 1\leq r\leq \red{s_{\rho}}.
\end{align*}
\end{Definition}
\end{tcolorbox}
In other words, \red{thinking of the case where $\rho$ is small}, $B$ is regular if it is an almost Euclidean isometry on sufficiently sparse \blue{unit} vectors, and it is strongly regular if \blue{additionally}, for any sufficiently sparse \blue{unit} vector $x$, $Bx$ is relatively well-spread: if $x$ is $r$-sparse, then \blue{$1-\rho\sqrt{r}\leq \|Bx\|_2^{\red{2}}\leq 1+\rho\sqrt{r}$} but the contribution of the $r$ largest coordinates \blue{of $Bx$ to its} Euclidean norm is at most \blue{$\rho \sqrt{r}$}.

To put this notion in perspective, consider the standard Gaussian matrix $A\in \R^{m\times n}$. \blue{Then for any} $u \geq 1$, with probability at least $1-2\exp(-c_0 u^2 r \log(em/r))$,
$$\sup_{x \in U_r}\left| \ \frac{1}{m}\|Ax\|_2^2 - \|x\|_2^2 \ \right|\leq c_1 u\sqrt{\frac{r\log(em/r)}{m}}$$
and
$$
\sup_{x\in \Sigma_{r,n}}\frac{1}{\sqrt{m}}\|Ax\|_{[r]}\leq c_1 u \sqrt{\frac{r\log(em/r)}{m}}.
$$
Hence, by taking the union bound over $1\leq r\leq m$ we find that with nontrivial probability, the Gaussian matrix $\frac{1}{\sqrt{m}}A$ is strongly $\red{\rho}$-regular for $\rho \sim \sqrt{\frac{\log \blue{m}}{m}}$.

\vskip0.3cm
The main result of this section is that by randomizing the columns of a strongly regular matrix $B$ using independent random signs, one obtains a matrix that is `well-behaved' on an arbitrary set. To formulate this claim, let $D_\eps$ be a diagonal matrix whose diagonal entries are independent, symmetric random signs $\eps_1,...,\eps_n$, and for $T \subset \R^n$ set
$$
d^*(T) = \left(\frac{\ell_*(T)}{{\cal R}(T)}\right)^2.
$$
\begin{tcolorbox}
\begin{Theorem}\label{thm:new}
There exists an absolute constant $c$ such that the following holds. Let $0<\rho<1/\sqrt{\log(m+n)}$ and consider $1 \leq k \leq m$. Assume that $B\in \R^{m\times n}$ is $\red{\rho}$-strongly regular. If $T\subset \R^n$, \red{$R\geq {\cal R}(T)$,} and $u \geq 1$ then with probability at least
$$
1-2\exp\left(-cu^2\left[\red{\ell_*^2(T)}+\red{R^2} k\log(em/k)\right]/\red{R^2}\right)
$$
with respect to $\eps$, we have that
\begin{align*}
& \sup_{x\in T} \left\| B  D_{\eps} x\right\|_{[k]}
\\
\leq & u^2 \left[ \rho \left(\ell_*(T)+ \red{R} \cdot \sqrt{k\log(em/k)}\right)  + \rho^2 \red{R}^{-1} \left(\ell_*^2(T)+ \red{R}^2 \cdot k\log(em/k)\right)\right].
\end{align*}
\end{Theorem}
\end{tcolorbox}

The proof of Theorem \ref{thm:new} is based on \blue{the following} analogous result on column randomization of a regular matrix, \blue{which was} established in \cite{Men21}.

\begin{Theorem} \label{thm:Men21}
There exist absolute constants $c,C>0$ such that the following holds. Let $0<\rho<1/\sqrt{\log n}$. Assume that $B \in \R^{m\times n}$ is $\red{\rho}$-regular. Then for any $T\subset \R^n$ and $u \geq 1$, with probability at least $1-2\exp(-cu^2d^*(T))$ with respect to $\eps$,
$$
\sup_{x\in T}\left|\|B D_{\eps} x\|_2^2 - \|x\|_2^2\right|\leq Cu^2{\cal R}^2(T) \cdot (\rho  \sqrt{d^*(T)} +\rho^2 d^*(T)).
$$
\end{Theorem}

The following simple lemma is the key to the proof of Theorem \ref{thm:new}. \blue{In what follows, $\operatorname{Id}_{m}\in \R^{m\times m}$ denotes the identity matrix.}

\begin{Lemma}\label{lem:new}
Let $B\in \R^{m\times n}$, $\rho>0$. If $B$ is $\red{\rho}$-strongly regular, then
 $\begin{bmatrix}
	B  & \operatorname{Id}_{m}
	\end{bmatrix}\blue{\in \R^{m\times (n+m)}}$
	is $\red{3\rho}$-regular.
\end{Lemma}	

\begin{proof}
Let $1\leq r\leq \red{\lceil(3\rho)^{-2}\rceil \leq \lceil\rho^{-2}\rceil}$. Clearly,
	\begin{align*}
	&\sup_{x\in \Sigma_{r,n}}\sup_{y\in \Sigma_{r,m}}\left|\left\|\begin{bmatrix}
	B & \operatorname{Id}_{m}
	\end{bmatrix}
	\begin{bmatrix}
	x \\
	y
	\end{bmatrix}
	\right\|_2^2 - \left\|\begin{bmatrix}
	x \\
	y
	\end{bmatrix}
	\right\|_2^2\right|
\\
	&\qquad \leq \sup_{x\in \Sigma_{r,n}} \left|  \|Bx\|_2^2 - \|x\|_2^2\right| + 2 \sup_{x\in \Sigma_{r,n}} \sup_{y\in \Sigma_{r,m}} |\inr{B x,y}|
\\
	&\qquad \leq \sup_{x\in \Sigma_{r,n}} \left|  \|B x\|_2^2 - \|x\|_2^2\right| + 2 \sup_{x\in \Sigma_{r,n}} \|B x\|_{[r]} \leq 3\rho\sqrt{r},
	\end{align*}
	and the result follows.
\end{proof}

\noindent{\bf Proof of Theorem \ref{thm:new}.} Set $R \geq {\cal R}(T)=\sup_{t \in T} \|t\|_2 $ and consider $T^\prime=\{t/R : t \in T\}$. We shall use \blue{the} fact that $\Sigma_{k,m}$ is an unconditional set: for any choice of signs $\zeta_1,...,\zeta_m$, $D_\zeta \Sigma_{k,m} = \Sigma_{k,m}$. \blue{This} allows one to introduce additional randomness. Indeed, let $\zeta_1,...,\zeta_m$ be independent, symmetric \red{random signs} that are also independent of $(\eps_i)_{i=1}^n$. Then
	\begin{equation}\label{eq:rescaled_set}
	\sup_{x\in T} \left\| B D_\eps x\right\|_{[k]}=R \sup_{x\in T^\prime} \left\| B D_\eps x\right\|_{[k]}
	\end{equation}
	and
	\begin{equation*}
	\sup_{x\in T^\prime} \left\| B D_\eps x\right\|_{[k]}  = \sup_{x\in T^\prime} \sup_{y \in \Sigma_{k,m}} |\inr{B D_\eps x,y}| =  \sup_{x\in T'} \sup_{y\in \Sigma_{k,m}} \left|\inr{ B D_\eps x, D_{\zeta} y}\right|.
	\end{equation*}
Moreover, by the polarization identity,
	\begin{align*}
	&  \left|\inr{ B D_\eps x, D_{\zeta}y}\right|
	 =  \frac{1}{4} \left| \left\|  B D_\eps x+D_{\zeta}y \right\|_2^2- \left\| BD_\eps x-D_{\zeta}y \right\|_2^2 \right|
\\
	& = \frac{1}{4}\left|\left\|
\begin{bmatrix}
	 BD_\eps  & D_{\zeta}
	\end{bmatrix}
	\begin{bmatrix}
	x \\
	y
	\end{bmatrix}
	\right\|_2^2
	-  \left\|
\begin{bmatrix}
	x \\
	y
	\end{bmatrix}
	\right\|_2^2
	 - \left(\left\|
\begin{bmatrix}
	 B D_\eps & D_{\zeta}
	\end{bmatrix}
	\begin{bmatrix}
	x \\
	-y
	\end{bmatrix}
	\right\|_2^2 - \left\|
\begin{bmatrix}
	x \\
	-y
	\end{bmatrix}
	\right\|_2^2\right) \right|.
	\end{align*}
	Clearly,
	$$\begin{bmatrix}
	 B D_{\eps} & D_{\zeta}
	\end{bmatrix} = \begin{bmatrix}
	 B  & \operatorname{Id}_{m}
	\end{bmatrix}
	\begin{bmatrix}
	D_{\eps} & 0 \\
	0 & D_{\zeta}
	\end{bmatrix}
	$$
	\blue{and hence}
	\begin{align}\label{eq:setting_of_thm:ORS}
	\sup_{x\in T^\prime} \left\| B D_{\eps}x\right\|_{[k]}
	\leq \frac{1}{2}\sup_{(x,y)^T\in \tilde{T}}\left|\left\|
\begin{bmatrix}
	 B & \operatorname{Id}_{m}
	\end{bmatrix}
	\begin{bmatrix}
	D_{\eps} & 0 \\
	0 & D_{\zeta}
	\end{bmatrix}
	\begin{bmatrix}
	x \\
	y
	\end{bmatrix}
	\right\|_2^2
	-  \left\|\begin{bmatrix}
	x \\
	y
	\end{bmatrix}
	\right\|_2^2\right|,
	\end{align}
	where $\tilde{T}=T^\prime\times \Sigma_{k,m}$.
	Lemma~\ref{lem:new} implies that
	$\begin{bmatrix}
	B  & \operatorname{Id}_{m}
	\end{bmatrix}$
is $\red{3\rho}$-regular, and one may therefore invoke Theorem~\ref{thm:Men21} for the set $\tilde{T}$. The result follows thanks to the straightforward observations that $\ell_*(\tilde{T})$ is equivalent to  $\ell_*(T^\prime)+\ell_*(\Sigma_{k,m})$; $\ell_*(\Sigma_{k,m}) \sim \sqrt{k\log(em/k)}$; and $\blue{{\cal R}(\tilde{T})=}\sup_{t \in \tilde{T}} \|t\|_2 \sim 1$.
\endproof

\begin{Remark} \label{rem:useful-presentation}
It is useful to present the estimate from Theorem~\ref{thm:new} in terms of the following parameter. For a set $T \subset \R^n$, $\red{R\geq  {\cal R}(T)}$, and $\red{1\leq k\leq m}$ define
\begin{equation}
\label{eqn:QkTdef}
\red{Q_k(T,R)=\rho\left(\ell_*^2(T)+R^2 k \log({em}/{k}) \right)^{1/2}}.
\end{equation}
Theorem \ref{thm:new} implies that with probability at least $1-2\exp(-c_0u^2[\red{\ell_*^2}(T)+\red{R^2}k\log(em/k)]/\red{R^2})$
\begin{equation}
\label{eqn:useful-presentation}
\sup_{\blue{x} \in T} \|BD_\eps x\|_{[k]} \leq c_1u^2 \max\{Q_k(T\red{,R}),\red{R}^{-1} Q_k^2(T\red{,R})\}.
\end{equation}
In particular, if $\red{Q_k(T,R) \leq R}$, as will be the case in the situations that interest us, the dominating term in the estimate from Theorem~\ref{thm:new} is $\sim u^2 Q_k(T\red{,R})$.
\end{Remark}

\subsection{Strong regularity and Theorem \ref{thm:mainGeneric}}
Let us return to the last two conditions that are required in the \blue{generic} embedding result from Theorem \ref{thm:mainGeneric}. If $T_\theta$ is a $\theta$-\red{net} of $T$ of minimal cardinality and $\red{k}=\lfloor \delta m/\lambda \rfloor$, one has to show that  
\begin{equation}
\label{eqn:knormbiasRepeat}
\sup_{x \in T_\theta} \|Ax\|_{[\red{k}]} \leq \lambda \sqrt{\red{k}}
\end{equation}
and
\begin{equation}
\label{eqn:knormoscillationsRepeat}
\sup_{x \in (T-T) \cap \theta B_2^n} \|Ax\|_{[\red{k}]} \leq \delta \sqrt{\red{k}}
\end{equation}
hold for the matrix
$$
A=R_I \Gamma_G D_{\eps^{\prime \prime}} \frac{1}{\sqrt{n}} \Gamma_{\eps^\prime} D_\eps.
$$
Set
$$
B= \frac{1}{\sqrt{\nrows}} R_I \Gamma_G D_{\eps^{\prime \prime}} \frac{1}{\sqrt{n}} \Gamma_{\eps^\prime}
$$
and observe that $A = \sqrt{\nrows} B D_\eps$. We will spend considerable effort in showing that $B$ is $\red{\rho}$-strongly regular, where $\red{\Upsilon:=}\rho \nrows^{1/2}$ is independent of $\nrows$ and is at most poly-logarithmic in $n$ \red{(see Theorem~\ref{thm:regularity-of-B})}. Before going down that long road, let us first show its benefit: \blue{combined with the following result, it will establish \eqref{eqn:knormbiasRepeat} and \eqref{eqn:knormoscillationsRepeat}.}   

\begin{tcolorbox}
\begin{Theorem} \label{thm:ugly-est}
For $u \geq 1$ there exist constants $c_1,c_2$ that depend only \blue{(polynomially)} on $u$ and absolute constants $\blue{c_3,}c_4$ such that the following holds. \blue{Let $0<\rho<\frac{1}{\sqrt{\log(m+n)}}$ be such that $\Upsilon= \rho \nrows^{1/2}  \geq 1$.} \blue{Let $B\in \R^{m\times n}$ be $\red{\rho}$-strongly regular,} set $0<\delta \leq {\cal R}(T)$ and consider
$$
\lambda\geq c_1\max\left\{\red{\delta}\Upsilon^2  \log(e \Upsilon), \Upsilon {\cal R}(T)\log^{1/2}\left(\frac{\Upsilon {\cal R}(T)}{\delta}\right) \right\}.
$$
Put
$$
0<\theta \leq \frac{\delta}{\blue{c_2}\Upsilon} \log^{-1/2}\left(\frac{\lambda}{\delta}\right).
$$
If
\begin{equation}
\label{eqn:ugly-estNumrows}
\nrows \geq c_3 \frac{\lambda}{\log\left(\frac{e\lambda}{\delta}\right)} \max\left\{ \frac{\red{\ell_*^2}(T_{\theta})}{\delta\red{{\cal R}^2(T)}},\frac{\red{\ell_*^2}((T-T)\cap \theta B_2^n)}{\delta\red{\theta^2}} \right\},
\end{equation}
then with probability at least $1-2\exp(-c_4u^2\red{m\delta/\lambda})$, \blue{the matrix $A=\sqrt{m}BD_{\eps}$ satisfies \eqref{eqn:knormbiasRepeat} and \eqref{eqn:knormoscillationsRepeat}.}
\end{Theorem}
\end{tcolorbox}

\begin{Remark}
To put the estimate \blue{\eqref{eqn:ugly-estNumrows}} in Theorem~\ref{thm:ugly-est} in a more familiar form, observe that 
$$
\red{\ell_*(T_\theta)\lesssim {\cal R}(T_\theta)\log |T_\theta| \lesssim {\cal R}(T)\log {\cal N}(T,\theta).}
$$
Hence, by choosing $\theta$ as large as possible, the \red{second term on the right hand side of \eqref{eqn:ugly-estNumrows}} leads to the $\ell_*^2/\delta^3$ term \blue{in \eqref{eqn:DoubleCirculantEmbeddingTextNumRows}} from Theorem \ref{thm:DoubleCirculantEmbeddingText}---up to poly-logarithmic factors in $n$, while the \red{first term} is dominated by the entropic \blue{term} \blue{in \eqref{eqn:DoubleCirculantEmbeddingTextNumRows}}.
\end{Remark}

\proof \blue{Set $\red{k}=\lfloor \delta m/\lambda \rfloor$ and} \red{define $K_{\theta}=(T-T) \cap \theta B_2^n$. We will apply the estimate in Remark~\ref{rem:useful-presentation} for the sets $T_{\theta}$ and $K_{\theta}$. To that end, observe that ${\cal R}(T_{\theta})\leq {\cal R}(T)$ and ${\cal R}(K_{\theta})\leq \theta$. Let use write $\red{Q(T_{\theta}) := Q_k(T_{\theta},{\cal R}(T))}$ and $\red{Q(K_\theta) := Q_k(T_{\theta},\theta)}$ and observe that
$$
\red{Q(T_{\theta}) \sim \rho  {\cal R}(T)\left(\frac{\ell_*^2(T_{\theta})}{ {\cal R}^2(T)}+ \frac{\delta}{\lambda} \nrows \log\left(\frac{e\lambda}{\delta}\right) \right)^{1/2}}
$$
and 
$$
\red{Q(K_\theta) \sim \rho \theta\left(\frac{\ell_*^2(K_{\theta})}{\theta^2}+ \frac{\delta}{\lambda} \nrows \log\left(\frac{e\lambda}{\delta}\right) \right)^{1/2}.}
$$
We begin by imposing conditions that ensure that 
\begin{equation}
\label{eqn:domCondRem}
Q(T_{\theta}) \leq {\cal R}(T), \ \ \ {\rm and} \ \ \ Q(K_\theta) \leq \theta,
\end{equation}
so that both $Q(T_{\theta})$ and $Q(K_{\theta})$ are the dominant terms in the estimate \blue{\eqref{eqn:useful-presentation}} for the sets $T_{\theta}$ and $K_{\theta}$, respectively. Observe that \eqref{eqn:ugly-estNumrows} implies that
$$
\max\left\{ \frac{\ell_*^2(T_{\theta})}{{\cal R}^2(T)} , \frac{\ell_*^2(K_{\theta})}{\theta^2}\right\} \leq \frac{\delta}{\lambda} \nrows \log\left(\frac{e\lambda}{\delta}\right),
$$
and hence \eqref{eqn:domCondRem} holds if we ensure that 
$$
\rho \sqrt{\nrows} \cdot \left(\frac{\blue{\delta}}{\blue{\lambda}}  \log\left(\frac{e\blue{\lambda}}{\blue{\delta}}\right) \right)^{1/2} \leq \frac{1}{4}.
$$
Since $\Upsilon = \rho \sqrt{\nrows}$, the latter condition is satisfied if}
\begin{equation} \label{eq:cond-lambda-delta-1}
\frac{\lambda}{\delta} \geq c_1 \Upsilon^2  \log(e \Upsilon).
\end{equation}
\blue{To summarize, if \eqref{eqn:ugly-estNumrows} and \eqref{eq:cond-lambda-delta-1} hold, then Remark~\ref{rem:useful-presentation} implies that
$$\sup_{x\in T_{\theta}} \|Ax\|_{[k]} = \sup_{x\in T_{\theta}} \sqrt{m}\|BD_{\eps} x\|_{[k]}\lesssim u^2\sqrt{m} Q(T_{\theta})$$
and
$$\sup_{x\in K_{\theta}} \|Ax\|_{[k]} \lesssim u^2\sqrt{m} Q(K_{\theta}).$$
Thus, to establish \eqref{eqn:knormbiasRepeat} and \eqref{eqn:knormoscillationsRepeat},} all that remains is to show that
$$
\blue{u^2}\sqrt{\nrows}Q(T_{\theta})\leq \lambda\sqrt{\red{k}} \ \ {\rm and} \ \  \blue{u^2}\sqrt{\nrows}Q(K_{\theta})\leq \delta\sqrt{\red{k}},
$$
i.e., that
\begin{equation} \label{eq:cond-Q-1}
\blue{u^2}\Upsilon {\cal R}(T) \left(\frac{\delta}{\lambda} \log \left(\frac{e \lambda}{\delta} \right) \right)^{1/2} \leq \sqrt{\delta} \sqrt{\lambda}
\end{equation}
and that
\begin{equation} \label{eq:cond-Q-2}
\blue{u^2}\Upsilon \theta \left(\frac{\delta}{\lambda} \log \left(\frac{e \lambda}{\delta} \right) \right)^{1/2} \leq \delta \sqrt{\frac{\delta}{\lambda}}.
\end{equation}
A straightforward computation shows that \eqref{eq:cond-Q-1} holds if
\begin{equation} \label{eq:cond-lambda-delta-2}
\lambda\geq \blue{c_1(u)} \Upsilon {\cal R}(T) \log^{1/2}\left(\frac{\Upsilon {\cal R}(T)}{\delta}\right),
\end{equation}
and \eqref{eq:cond-Q-2} holds if
\begin{equation} \label{eq:cond-theta-1}
\theta \leq \frac{\delta}{\blue{c_2(u)}\Upsilon} \log^{-1/2}\left(\frac{\blue{e}\lambda}{\delta}\right).
\end{equation}
This completes the proof. 
\endproof

\section{Strong regularity features of the double circulant matrix} \label{sec:reg-double-circulant}
\blue{Recall that} the double circulant matrix is
$$
A=R_I \Gamma_G D_{\eps^{\prime \prime}} \frac{1}{\sqrt{n}}\Gamma_{\eps^\prime} D_\eps = \sqrt{m} B D_{\eps},
$$
where $I \subset  \{1,...,n\}$ is a fixed set of indices of cardinality $m$, $G$ is the standard gaussian random vector in $\R^n$ and $\eps$, $\eps^{\prime}$ and $\eps^{\prime \prime}$ are uniformly distributed in $\{-1,1\}^n$. Moreover, $G$, $\eps$,  $\eps^{\prime}$ and $\eps^{\prime \prime}$ are all independent. Also, for every $x \in \R^n$, we have that
$$
\Gamma_G x = G \circledast x = \Gamma_x G
$$
is the discrete circular convolution of $G$ and $x$. And, denoting by ${\cal F}\in \blue{\C}^{n\times n}$ the discrete Fourier matrix, we have that  $\Gamma_ x G= \sqrt{n}UD_{Wx}OG$, where
\begin{equation}
\label{eqn:OWUDef}
O=W=\frac{{\cal F}}{\sqrt{n}}, \ \ {\rm and} \ \ U=\frac{{\cal F}^{-1}}{\sqrt{n}}=W^*.
\end{equation}
In particular, $U$, $W$ and $O$ are \emph{Hadamard-type matrices}, i.e., they are unitary and all their entries are bounded by $\frac{1}{\sqrt{n}}$.

\vskip0.3cm
In light of Theorem \ref{thm:mainGeneric}, a key part of the analysis of the embedding procedure is to show that $A$ maps an arbitrary set $T$ into `well-spread vectors', specifically, that \eqref{eqn:knormbiasRepeat} and \eqref{eqn:knormoscillationsRepeat} hold. By invoking \blue{Theorem~\ref{thm:ugly-est}}, that can be established by proving that $B$ \blue{possesses the following strong regularity property:}
\begin{tcolorbox}
\begin{Theorem} \label{thm:regularity-of-B}
There is an absolute constant $c_0$ and for $\gamma \geq 1$ \blue{there are constants $c_1$ and $c_2$} \red{that depend only (polynomially) on $\gamma$} such that the following holds. \blue{If $\nrows \leq n/(\red{c_1}\log^{\red{4}} n)$}, then with probability at least $1-2\exp(-c_0\gamma \log^2 n)$ with respect to $G \otimes \eps^\prime \otimes \eps^{\prime \prime}$, the random matrix $B$ is $\red{\rho}$\blue{-}strongly regular for
$$
\rho=\red{c_2}\frac{\log^{\red{5/2}} n}{\sqrt{\nrows}}.
$$
\end{Theorem}
\end{tcolorbox}
Note \red{that} $\blue{\Upsilon=}\rho\sqrt{m}$ is poly-logarithmic in $n$, as required.\par
The idea behind the proof of Theorem~\ref{thm:regularity-of-B} is outlined in the next section. The \red{full} proof is presented in \blue{Appendix}~\ref{sec:strong-regularity-double-circ-proof}.

\subsection{Highlights of the proof of Theorem~\ref{thm:regularity-of-B}}
The \blue{proof is} based on two well-known facts that are formulated in what follows. The first fact is an outcome of \cite{KMR14}, on the behavior of second-order chaos processes (see also Theorem~6.5 in \cite{Dir15} for the refinement that is used here). The bound is based on Talagrand's $\gamma_2$-functional. For a detailed exposition on chaining methods and the $\gamma$-functionals in a general setup, we refer the reader to \cite{Tal14}.

\begin{Definition} \label{def:gamma-2}
Let ${\cal A}$ be a subset of a normed space. An admissible sequence of ${\cal A}$ is a collection of sets ${\cal A}_{\red{\ell}} \subset {\cal A}$, where $|{\cal A}_0|=1$ and for $\ell \geq 1$, $|{\cal A}_{\red{\ell}}| \leq 2^{2^{\red{\ell}}}$. For $a \in {\cal A}$, let $\pi_{\red{\ell}} a \in {\cal A}_{\ell}$ be a nearest point to $a$ in ${\cal A}_{\red{\ell}}$ with respect to the underlying norm. Define
$$
\gamma_2({\cal A},\| \cdot \|) = \inf \sup_{a \in {\cal A}} \left( \|\pi_0 a\| + \sum_{{\red{\ell}} \geq 1} 2^{{\red{\ell}}/2}\|\pi_{{\red{\ell}}}a -\pi_{{\red{\ell}}-1}a\| \right),
$$
where the infimum is taken with respect to all admissible sequences of ${\cal A}$.
\end{Definition}
In the setup we focus on here, ${\cal A}$ is a class of matrices. Denote by $\| \cdot \|_{2 \to 2}$ the standard operator norm, let $\| \cdot \|_{HS}$ be the Hilbert-Schmidt norm and put
$$
d_{HS}({\cal A})=\sup_{A \in {\cal A}} \|A\|_{HS}, \ \ \  \ \ \ d_{2 \to 2}({\cal A}) = \sup_{A \in {\cal A}} \|A\|_{2 \to 2}.
$$
Let $\gamma_2({\cal A}) \equiv \gamma_2({\cal A},\| \cdot \|_{2 \to 2})$ be the $\gamma_2$-functional of ${\cal A}$ with respect to the operator norm.

\vskip0.3cm
Recall that a \red{centered} random variable $\xi$ is $L$-subgaussian if for every $p \geq 2$, $\|\xi\|_{L_p} \leq L\sqrt{p} \|\xi\|_{L_2}$. A random vector $X$ is $L$-subgaussian if it is \red{centered} and for any $t \in \R^n$, $\inr{X,t}$ is an $L$-subgaussian random variable.
\begin{Theorem} \label{thm:KMR14}
Let $\xi$ be a random vector whose coordinates $(\xi_i)_{i=1}^n$ are independent, mean-zero, variance $1$ random variables that are also $L$-subgaussian. Then for any $u \geq 1$, with probability at least $1-2\exp(-u)$,
$$
\sup_{A \in {\cal A}} \left| \ \|A\xi\|_2^2 - \E \|A\xi\|_2^2 \ \right| \leq C(L) \left(\gamma_2^2({\cal A}) + d_{HS}({\cal A})\gamma_2({\cal A}) + \sqrt{u} d_{HS}({\cal A})d_{2\to 2}({\cal A})+ud^2_{2\to 2} ({\cal A}) \right)
$$
for a constant $C\red{(L)}$ that depends only on $L$.
\end{Theorem}
To see why Theorem~\ref{thm:KMR14} is useful for establishing regularity, let us return to the random operator
$$
B=\frac{1}{\sqrt{\nrows}} R_I \Gamma_G D_{\eps^{\prime \prime}} \frac{1}{\sqrt{n}} \Gamma_{\eps^\prime}.
$$
Set
$$
\Psi= D_{\eps^{\prime \prime}} \frac{1}{\sqrt{n}} \Gamma_{\eps^\prime},
$$
\blue{then}, for every $x \in \R^n$,
$$
Bx=\frac{1}{\sqrt{\nrows}} R_I \Gamma_G \Psi x = \frac{1}{\sqrt{\nrows}} R_I \Gamma_{\Psi x} G.
$$
Therefore,
\begin{equation} \label{eq:RIP-B-1}
\sup_{x \in \Sigma_{r,n}} \left| \ \|Bx\|_2^2 - \|x\|_2^2 \ \right| \leq \sup_{x \in \Sigma_{r,n}} \left| \ \|Bx\|_2^2 - \E_G \|Bx\|_2^2 \ \right| + \sup_{x \in \Sigma_{r,n}} \left| \ \E_G \|Bx\|_2^2 - \|x\|_2^2 \ \right|,
\end{equation}
and by a straightforward computation,
\begin{equation} \label{eq:exp-B}
\E_G \|Bx\|_2^2 = \E_G  \left\|\frac{1}{\sqrt{\nrows}} R_I \Gamma_{\Psi x} G\right\|_2^2 = \frac{1}{n} \|\Gamma_{\eps^\prime} x \|_2^2 \blue{= \frac{1}{n} \|\Gamma_{x} \eps^\prime \|_2^2}.
\end{equation}
Hence, \eqref{eq:RIP-B-1} becomes
\begin{align} \label{eq:RIP-B-2}
& \sup_{x \in \Sigma_{r,n}} \left| \ \|Bx\|_2^2 - \|x\|_2^2 \ \right| \nonumber
\\
\leq & \sup_{x \in \Sigma_{r,n}}  \left| \ \left\| \frac{1}{\sqrt{\nrows}}R_I \Gamma_{\Psi x} G \right\|_2^2 - \E_G \left\|\frac{1}{\sqrt{\nrows}} R_I \Gamma_{\Psi x} G \right\|_2^2 \ \right| + \sup_{x \in \Sigma_{r,n}} \left| \ \left\|\frac{1}{\sqrt{n}}\blue{\Gamma_{x} \eps^\prime} \right\|_2^2 - \|x\|_2^2 \ \right|
\end{align}
and
$$
\|x\|_2^2 = \E_{\eps^\prime} \left\|\frac{1}{\sqrt{n}} \Gamma_x \eps^\prime\right\|_2^2.
$$
The two terms in \eqref{eq:RIP-B-2} are exactly in the form that is dealt with in Theorem \ref{thm:KMR14}. Taking into account the regularity estimates we are looking for, the classes of matrices of interest are
\begin{equation} \label{eq:class-for-B-1}
\left\{ \blue{\frac{1}{\sqrt{m}}}R_I \Gamma_{\Psi x} : x \in \Sigma_{r,n} \right\} \ \ {\rm for} \  r \leq \nrows,
\end{equation}
and
\begin{equation} \label{eq:class-for-B-2}
\left\{\blue{\frac{1}{\sqrt{n}}}\Gamma_x : x \in \Sigma_{r,n} \right\} \ \ {\rm for}  \ r \leq n.
\end{equation}

The key estimate in the case of \eqref{eq:class-for-B-2} was established in \cite{KMR14}:
\begin{Theorem} \label{thm:RIP-sparse-simple}
For $L, \gamma \geq 1$ there are constants $c_0$ and $c_1$ \blue{depending only on $L$ and $\gamma$, respectively,} such that the following holds. Let $\xi$ be as in Theorem \ref{thm:KMR14} and set $1 \leq r \leq n$. Then with probability at least $1-2\exp(-c_0\gamma \log^4 n)$,
$$
\sup_{x \in \Sigma_{r,n}} \left| \ \red{\frac{1}{n}}\|\Gamma_x \xi\|_2^2 - \red{\|x\|_2^2} \ \right| \leq \rho \sqrt{r}
$$
for
$$
\rho = c_1(\gamma) \frac{\log^{\red{2}} n}{\sqrt{n}}.
$$
\end{Theorem}
Obtaining a similar estimate for the class \eqref{eq:class-for-B-1} is technically more involved but is based on similar ideas. The details are presented in Appendix \ref{sec:strong-regularity-double-circ-proof}.

\vskip0.3cm
Next, one has to establish the strong regularity of $B$. That is based on the following fact, which is a straightforward generalization of Theorem~3.4 from \cite{DiM18b}.  To formulate the claim, consider two unitary matrices ${\cal O}, {\cal U}\in \C^{n\times n}$, and let ${\cal W}\in \C^{n\times n}$. Set
$$
d_{\cal U} = \sqrt{n} \max_{1 \leq i,j \leq n} |{\cal U}_{ij}| \ \ \ {\rm and} \ \ \ d_{\cal W} = \sqrt{n} \max_{1 \leq i,j \leq n} |{\cal W}_{ij}|,
$$
and assume further that
$$
\sup_{x \in \Sigma_{r,n}} \|{\cal W}x\|_2 \leq 2.
$$
Note, for example, that if ${\cal U}$ and ${\cal W}$ are Hadamard-type matrices then $d_{\cal U}, d_{\cal W} \leq 1$ and the condition on ${\cal W}$ is trivially satisfied.

\begin{Theorem} \label{thm:s-normEstimates-sparse}
For $L \geq 1$ there exist constants $c_1$ and $c_2$ that depend only on $L$ such that the following holds. Let $\zeta$ be an $L$-subgaussian random vector, set $1 \leq r \leq n$, and let $u\geq 1$. Then with probability at least $1-e^{-c_1u}$,
\begin{equation*}
\sup_{x\in \Sigma_{r,n}}\|{\cal U} D_{{\cal W} x} {\cal O} \zeta\|_{[r]} \leq c_2 \sqrt{\frac{r}{n}} \max\{d_{{\cal U}},d_{{\cal W}}\}\left(\log(n)\log(r) +  \sqrt{u}\right).
\end{equation*}
\end{Theorem}
\blue{The proof of Theorem~\ref{thm:s-normEstimates-sparse} is based on a straightforward modification of the argument used to prove Theorem~3.4 in \cite{DiM18b}; its details are omitted.}

\vskip0.3cm

Let us return to the random matrix $B$. To establish $\red{\rho}$-strong regularity, one has to estimate
$\sup_{x \in \Sigma_{r,n}} \|Bx\|_{[r]}$ for every $1 \leq r \leq m$. Since
$$
Bx= \frac{1}{\sqrt{\nrows}} R_I \Gamma_{\Psi x} G,
$$
and  $\|R_I \Gamma_{\Psi x} G\|_{[r]} \leq \|\Gamma_{\Psi x} G\|_{[r]}$, it suffices to control
$$
\|\Gamma_{\Psi x} G\|_{[r]} = \sqrt{n}\|UD_{W \Psi x} OG\|_{[r]}
$$
for \red{the} Hadamard-type matrices $U$, $W$, and $O$ \red{defined in \eqref{eqn:OWUDef}}. Thus, a high probability (with respect to $G$) estimate on $\sup_{x \in \Sigma_{r,n}} \|U D_{W \Psi x} OG\|_{[r]}$ follows from Theorem \ref{thm:s-normEstimates-sparse} once one shows that
\begin{equation} \label{eq:Psi-ok-sparse}
\sup_{x \in \Sigma_{r,n}} \|W \Psi x\|_2 = \sup_{x \in \Sigma_{r,n}} \left\| \frac{1}{\sqrt{n}} D_{\eps^{\prime \prime}} \Gamma_{\eps^\prime} x \right\|_2 \leq 2.
\end{equation}
The proof of \eqref{eq:Psi-ok-sparse} is straightforward, thanks to Theorem \ref{thm:RIP-sparse-simple}. Indeed,
$$
\left\| \frac{1}{\sqrt{n}} D_{\eps^{\prime \prime}} \Gamma_{\eps^\prime} x \right\|_2 = \left\| \frac{1}{\sqrt{n}}\Gamma_{x} \eps^\prime\right\|_2,
$$
and by Theorem \ref{thm:RIP-sparse-simple}, with high probability with respect to $\eps^\prime$,
$$
\sup_{x \in \Sigma_{r,n}} \left| \ \|\Gamma_x \eps^\prime \|_2^2 - \E \|\Gamma_x \eps^\prime \|_2^2 \ \right|
$$
is well-behaved.

\section{Concentration of the Gaussian convolution operator}
\label{sec:conv}

Next, let us turn to the \blue{second ingredient needed for the application of the generic embedding result, formulated in Theorem~\ref{thm:mainGeneric}:} proving the $\ell_1$-concentration phenomenon for the double circulant matrix.

Let $G$ be the standard gaussian random vector in $\R^n$ and for a fixed $x \in \R^n$ and $I \subset \{1,...,n\}$, consider the partial convolution $R_I(x \circledast G)$. The first order of business is to study the concentration of $\|R_I(x \circledast G)\|$ around its mean. An important ingredient in the analysis is the following immediate consequence of the Gaussian concentration theorem for Lipschitz functions (see, e.g., \cite{LT91}).
\begin{Theorem} \label{thm:gaussian-conc}
\blue{Let $S\subset \R^n$ and set ${\cal R}(S) = \sup_{x \in S} \|x\|_2$. Then for $u>0$,
$$
\PP \left( \left| \ \sup_{x\in S} \langle G,x\rangle-\E\sup_{x\in S} \langle G,x\rangle \ \right| \geq u {\cal R}(S) \right) \leq 2\exp(-cu^2),
$$
where $c$ is an absolute constant.}
\end{Theorem}

To formulate the concentration estimate for convolutions, recall that $W={\cal F}/\sqrt{n}$, where $\cF$ is the discrete Fourier matrix.
\begin{tcolorbox}
\begin{Theorem} \label{thm:conc-conc}
There is an absolute constant $\red{c>0}$ such that the following holds. \blue{Let $\| \cdot \|$ be a norm on $\R^n$, let ${\cal B}$ be the corresponding unit ball, and denote by ${\cal B}^\circ$ the dual unit ball.} For any $I \subset \{1,...,n\}$ and $u>0$, with probability at least $1-2\exp(-cu^2)$,
$$
\left| \ \|R_I(x \circledast G) \| - \E \|R_I(x \circledast G) \| \ \right| \leq \sqrt{n} {\cal R}(\red{R_I} {\cal B}^\circ) \inf_{\{y,z : x =y+z\}} \left(u \|Wz\|_\infty + 2 \|WG\|_\infty \|y\|_2 \right).
$$
\end{Theorem}
\end{tcolorbox}
\proof
Fix a decomposition $x=y+z$ and observe that
$$
x \circledast G = y \circledast G + z \circledast G = y \circledast G + \sqrt{n} UD_{Wz}OG,
$$
where, as before, $O=W$ and $U=W^*$. Clearly,
$$
y \circledast G = W^* W (y \circledast G)=\sqrt{n} W^*\left(\sum_{i=1}^n (Wy)_i (WG)_i e_i \right),
$$
\red{$R_I^*=R_I$}, and hence, almost surely,
\begin{align}
\|R_I (y \circledast G)\| = & \blue{\sqrt{n}}\left\| R_I W^*\left(\sum_{i=1}^n (Wy)_i (WG)_i e_i \right) \right\| \nonumber \\
\leq & \blue{\sqrt{n}}\sup_{a \in B_2^n} \|R_I W^* a\| \cdot \left\|\sum_{i=1}^n (Wy)_i (WG)_i e_i  \right\|_2\nonumber
\\
\leq & \blue{\sqrt{n}}{\cal R}(\red{R_I} {\cal B}^\circ) \cdot \|WG\|_\infty \|y\|_2; \label{eq:y-term-convolution}
\end{align}
in particular, th\red{is} estimate \red{also} holds for $\E \|R_I(y \circledast G)\|$.

Next, observe that
$$
\blue{\|z \circledast G\|} = \left\|\sqrt{n} R_I UD_{Wz}OG \right\| = \sqrt{n} \sup_{t \in B^\circ} \inr{G,O^*D_{Wz}^*U^*\red{R_I^*}t},
$$
and
$$
\sup_{t \in B^\circ} \blue{\|O^*D_{Wz}^*U^*R_I^* t\|_2} \leq \|Wz\|_\infty \cdot {\cal R}(\red{R_I} {\cal B}^\circ).
$$
Hence, by Theorem~\ref{thm:gaussian-conc}, for $u>0$, with probability at least $1-2\exp(-cu^2)$,
\begin{equation} \label{eq:z-term-convolution}
\left|  \ \|z \circledast G\| - \E \|z \circledast G\| \ \right| \leq u \sqrt{n} \cdot {\cal R}(\blue{\red{R_I}} {\cal B}^\circ) \cdot \|Wz\|_\infty.
\end{equation}
The claim follows by combining \eqref{eq:y-term-convolution} and \eqref{eq:z-term-convolution}.
\endproof

\blue{To apply this result for} the $\ell_1^n$-norm, \blue{note that} for every $x \in \R^n$,
$$
\E \|R_I (x \circledast G)\|_1 = \E|g| \cdot \nrows \|x\|_2  = \sqrt{\frac{\blue{2}}{\blue{\pi}}} \nrows \|x\|_2,
$$
and ${\cal R}(\red{R_I} {\cal \blue{B}}^\circ) = \sup_{t \in B_\infty^I} \|t\|_2 = \sqrt{\nrows}$. Hence, for $u>0$, with probability at least $1-2\exp(-cu^2)$,
\begin{equation} \label{eq:ell-1-conc}
\left| \ \left\|R_I(x \circledast G) \right\|_1 - \sqrt{\frac{\blue{2}}{\blue{\pi}}} \nrows \|x\|_2 \ \right| \leq \sqrt{n \nrows} \inf_{x=y+z} \left(u \|Wz\|_\infty + 2 \|WG\|_\infty \|y\|_2 \right).
\end{equation}

With \eqref{eq:ell-1-conc} in mind, let us show that when the Fourier transform of $x$ is well-spread, $\|R_I (x \circledast G)\|_1$ exhibits sharp concentration around its mean. One useful notion of `being well-spread' is that there is some $1 \leq r \leq n$ and $1 \leq \Lambda \leq \sqrt{n}$ such that
\begin{equation} \label{eq:well-spread}
\|W x\|_{[r]} \leq \frac{\Lambda}{\sqrt{n}} \|x\|_2.
\end{equation}

\begin{Corollary} \label{cor:ell-1-conc-good-position}
There is an absolute constant $\red{c>0}$ such that the following holds. Let $x$ satisfy \eqref{eq:well-spread}. For any $\delta>0$, with probability at least
$$
1-2\exp\left(-c \delta^2 \frac{\nrows r}{\Lambda^2}\right),
$$
we have that
\begin{equation} \label{eq:ell-1-conc-good-position-1}
\left| \ \left\|R_I(x \circledast G) \right\|_1 - \sqrt{\frac{\blue{2}}{\blue{\pi}}} \nrows \|x\|_2 \ \right| \leq \sqrt{\frac{\blue{2}}{\blue{\pi}}} \nrows \|x\|_2  \left(\delta + \sqrt{\blue{2\pi}} \frac{\Lambda}{\sqrt{\nrows}} \|WG\|_\infty\right).
\end{equation}
\end{Corollary}

\proof Let $J$ be  the set of indices corresponding to the largest $r$ coordinates of $(|(Wx)_i|)_{i=1}^n$. Using the invertibility of $W$, one may write $x=y+z$ where $Wy=R_J Wx$ and $W z=R_{J^c} Wx$. \blue{By \eqref{eq:well-spread},}
$$
\|y\|_2 \red{=\|Wy\|_2} = \|Wx\|_{[r]} \leq \frac{\Lambda}{\sqrt{n}} \|x\|_2, \ \ {\rm and} \ \ \|Wz\|_\infty \leq \frac{1}{\sqrt{r}}\|Wx\|_{[r]}\leq \frac{\Lambda}{\sqrt{rn}} \|x\|_2.
$$
Hence, \eqref{eq:ell-1-conc} implies that with probability at least $1-2\exp(-cu^2)$,
\begin{align} \label{eq:ell-1-conc-2}
\left|  \ \left\|R_I(x \circledast G) \right\|_1 - \sqrt{\frac{\blue{2}}{\blue{\pi}}} \nrows \|x\|_2 \ \right| \leq & \sqrt{n \nrows} \left(u \|Wz\|_\infty + 2 \|WG\|_\infty \|y\|_2 \right) \nonumber
\\
\leq & \|x\|_2 \sqrt{\nrows} \left(u \frac{\Lambda}{\sqrt{r}} + 2 \|WG\|_\infty \Lambda \right).
\end{align}
The claim follows by setting $u=\sqrt{\frac{\blue{2}}{\blue{\pi}}} \delta (\sqrt{\nrows r}/\Lambda)$.
\endproof

It is straightforward to verify that a similar estimate holds if instead of \eqref{eq:well-spread} we only have an upper estimate on $\|Wx\|_{[r]}$. That is the form we use in what follows.

\begin{tcolorbox}
\begin{Corollary} \label{cor:ell-1-conc-good-position-2}
There is an absolute constant $c$ such that the following holds. Let $x$ satisfy that $\|W x\|_{[r]} \leq \frac{\Lambda}{\sqrt{n}}$. Then for $\delta>0$, with probability at least
$$
1-2\exp\left(-c \delta^2 \frac{\nrows r}{\Lambda^2}\right),
$$
we have that
\begin{equation} \label{eq:ell-1-conc-good-position-2}
\left| \ \left\|R_I(x \circledast G) \right\|_1 -  \sqrt{\frac{\blue{2}}{\blue{\pi}}} \nrows \|x\|_2 \ \right| \leq \sqrt{\frac{\blue{2}}{\blue{\pi}}} \nrows  \left(\delta + \sqrt{\blue{2\pi}} \frac{\Lambda}{\sqrt{\nrows}} \|WG\|_\infty\right).
\end{equation}
\end{Corollary}
\end{tcolorbox}

We will apply Corollary~\ref{cor:ell-1-conc-good-position-2} to a finite set that is in a `good position'---namely, consisting of vectors for which there is enough control on $\|W \cdot \|_{[r]}$.
\vskip0.3cm
To put Corollary~\ref{cor:ell-1-conc-good-position-2} in some perspective, let us consider our benchmark---the Gaussian matrix.

\subsection{The gaussian benchmark}

Let $\Gamma = n^{-1/2}\sum_{i=1}^n \inr{G_i,\cdot} e_i$ be the normalized standard Gaussian matrix. Consider $T \subset \R^n$ and let us explore the outcome of Corollary \ref{cor:ell-1-conc-good-position-2} for points in the set $\Gamma T$. To that end, \blue{pick} \red{the largest $1\leq r\leq n$} such that
$$
r\log \left(\frac{\red{e}n}{r}\right) \leq \left(\frac{\ell_*(T)}{{\cal R}(T)} \right)^2 = d^*(T).
$$
\red{Assuming that $d^*(T)\geq \log n$, $r$ is well-defined and satisfies $r \sim d^*(T)/\log(en/d^*(T))$. By the proof of \cite{DMS21a}, Theorem 2.5, for any $u\geq 1$, we have}
\begin{equation} \label{eq:gaussian-rearrangement}
\|\red{W}\Gamma t\|_{[r]} \leq C u\frac{\ell_*(T)}{\sqrt{n}},
\end{equation}
with probability at least $1-2\exp(-cu^2 d^*(T))$. Denote that event by $\Omega_u$, and let us invoke Corollary \ref{cor:ell-1-conc-good-position-2}, where for every $t \in T$ we set
$$
r \sim \frac{d^*(T)}{\log(en/d^*(T))}, \ \ \ {\rm and} \ \ \ \Lambda \sim \ell_*(T).
$$
It follows that conditioned on $\Omega_u$, for $\delta>0$ and every $t \in T$, with probability at least
$$
1-2\exp\left(-c_0\delta^2 \frac{\nrows  d^*(T)}{\ell_*^2(T) \log(en/d^*(T))}\right)=1-2\exp\left(-c_0\delta^2 \frac{\nrows}{{\cal R}^2(T) \log(en/d^*(T))}\right)
$$
with respect to $G$ we have that
$$
\left|  \ \left\|R_I(\Gamma t \circledast G) \right\|_1 - \blue{\sqrt{\frac{2}{\pi}}} \nrows \blue{\|\Gamma t\|_2}  \ \right| \leq \blue{c_1} \nrows   \left(\delta + \frac{\ell_*(T)}{\sqrt{\blue{\nrows}}} \|WG\|_\infty\right).
$$

In the next section we will show that the double circulant matrix $A$ satisfies an almost identical inequality.

\section{Uniform $\ell_1$-concentration for the double circulant matrix} \label{sec:double-circulant-ell-1}

Fix $1 \leq r \leq n$. Corollary \ref{cor:ell-1-conc-good-position-2} implies that for any $t \in \R^n$, if
\begin{equation} \label{eq:ell-1-conc-A-cond}
\blue{\frac{1}{\sqrt{n}}} \|W D_{\eps^{\prime \prime}}\Gamma_{\eps^\prime} D_\eps t\|_{[r]} \leq \frac{\Lambda}{\sqrt{n}},
\end{equation}
then with high probability with respect to $G$, $\|At\|_1$ concentrates around its mean. Thus, the key question is whether, with high probability, for every $t \in T$, \eqref{eq:ell-1-conc-A-cond} holds for suitable values of $\Lambda$ and $r$. We will show that with high probability, one may set
\begin{equation}
\label{eq:ell-1-conc-A-condParam}
r \sim \frac{d^*(T)}{\log\left(\frac{en}{d^*(T)}\right)}, \ \ \ {\rm and} \ \ \ \Lambda \sim \Upsilon \ell_*(T), \text{ \blue{where} } \blue{\Upsilon \sim \log^{\red{5/2}}n}.
\end{equation}
Up to the factor of $\Upsilon$, this is the same as in the gaussian case.
\begin{Theorem} \label{thm:ell-1-conc-double circulant}
For $\gamma \geq 1$ there are constants $c_0,c_1$ and $\red{c_2}$ and an event $\Omega_1$ with probability at least $1-n^{-\gamma}$ with respect to $\eps^\prime \otimes \eps^{\prime \prime}$ such that the following holds. Let $T \subset \R^n$ and assume that $\red{\log n}\leq d^*(T) \leq n/\log^{\red{5}} n$. There is an event $\Omega_{2,T}$ with probability at least $1-2\exp(-c_0 u^2 d^*(T))$ with respect to $\eps$, such that, conditioned on $\Omega_1$ and $\Omega_{2,T}$, for every $t \in T$, with probability at least
$$
1-2\exp\left(-c_1\delta^2 \frac{\nrows}{{\cal R}^2(T)\Upsilon^2 \log\left(\frac{en}{d^*(T)}\right)}\right)
$$
with respect to $G$,
\begin{equation} \label{eq:ell-1-conc-A-1}
\left| \ \frac{1}{\nrows} \| A t\|_1 - \sqrt{\frac{2}{\pi}} \red{\frac{1}{\sqrt{n}}\|\Gamma_{\eps'}D_{\eps} t\|_2} \ \right| \leq c_2 \left(\delta + \frac{\|WG\|_\infty \Upsilon \ell_*(T)}{\sqrt{\nrows}} \right),
\end{equation}
where $\Upsilon \sim \log^{\red{5/2}} n$.
\end{Theorem}

Therefore, up to the factors of $\Upsilon$ in the probability estimate and in \eqref{eq:ell-1-conc-A-1}, Theorem \ref{thm:ell-1-conc-double circulant} shows that the $\ell_1$-concentration phenomenon exhibited by the double circulant matrix is the same as exhibited by the standard gaussian matrix. The proof of Theorem~\ref{thm:ell-1-conc-double circulant} is presented in the next section.

\begin{tcolorbox}
\begin{Corollary} \label{cor:ell-1-conc-double circulant}
\red{Let $\gamma\geq 1$. There exists a constant $c_1$ depending only (polynomially) on $\gamma$ such that the following holds.} Let $T^\prime$ be a finite set. If $\log n \leq d^*(T^\prime) \leq c_0n/\log^{\red{5}} n$, \red{$m\leq n$}, and
$$
\nrows \geq \red{c_1} {\cal R}^2(T^\prime)\frac{\log |T^\prime|}{\delta^2} \cdot \log^{\red{6}} n
$$
then with probability at least $1-n^{-\gamma}$ with respect to $G \otimes \eps \otimes \eps^\prime \otimes \eps^{\prime \prime}$,
$$
\sup_{t \in T^\prime} \left| \ \frac{1}{\nrows} \| A t\|_1 - \sqrt{\frac{2}{\pi}} \| t\|_2 \ \right| \leq \red{\delta}.
$$
\end{Corollary}
\end{tcolorbox}
\proof \red{By Hoeffding's inequality and a union bound, with probability at least $1-n^{-\gamma}$ with respect to $G$,} 
$$\red{\|WG\|_{\infty} \leq c_0 \sqrt{\gamma \log n}.}$$
\red{Hence,} \blue{by Theorem~\ref{thm:ell-1-conc-double circulant}, we only need to establish uniform concentration of $\frac{1}{\sqrt{n}}\|\Gamma_{\eps'} D_{\eps}t\|_2$ around $\|t\|_2$ for all $t\in T'$. Set $\tilde{T}=\{t/\|t\|_2 \ : \ t\in T^\prime\}$. Using that $|a-1| = \frac{|a^2-1|}{a+1}\leq |a^2-1|$ if $a\geq 0$, we find
\begin{align*}
\sup_{t\in T^\prime} \Big|\frac{1}{\sqrt{n}}\|\Gamma_{\eps'} D_{\eps}t\|_2-\|t\|_2\Big| & = \sup_{t\in T^\prime} \|t\|_2\Big|\frac{1}{\sqrt{n}}\Big\|\Gamma_{\eps'} D_{\eps}\frac{t}{\|t\|_2}\Big\|_2-1\Big|\\
& \leq {\cal R}(T^\prime)\ \sup_{t\in \tilde{T}} \Big|\frac{1}{n}\|\Gamma_{\eps'} D_{\eps}t\|_2^2-1\Big|.
\end{align*}
By Theorem \ref{thm:RIP-sparse-simple}, with probability at least $1-2\exp(-c_1\gamma \log^4 n)$ with respect to $\eps^\prime$, the matrix $n^{-1/2}\Gamma_{\eps^\prime}$ is $\red{\rho}$-regular for
\begin{equation*} 
\rho =c_2(\gamma)\frac{\log^{\red{2}} n}{\sqrt{n}}.
\end{equation*}
On that event, Theorem~\ref{thm:Men21} implies that with probability at least $1-\exp(-u^2\log |T^\prime|)$ with respect to $\eps$, 
\begin{align*}
\sup_{t\in \tilde{T}} \Big|\frac{1}{n}\|\Gamma_{\eps'} D_{\eps}t\|_2^2-1\Big| & \leq c_3u^2{\cal R}(\tilde{T})\left(\rho\sqrt{d^*(\tilde{T})}+\rho^2 d^*(\tilde{T})\right)\\
& \sim c_4(\gamma)u^2\left(\frac{\log^{\red{2}}(n)\sqrt{\log |T^\prime|}}{\sqrt{n}} + \frac{\log^{\red{4}}(n)\log |T^\prime|}{n}\right),
\end{align*}
as ${\cal R}(\tilde{T})=1$ and $d^*(\tilde{T}) = \ell_*^2(\tilde{T})\lesssim \log |T^\prime|$. Setting $u^2=c_5\gamma\log(n)$, we conclude that 
$$\sup_{t\in T^\prime} \Big|\frac{1}{\sqrt{n}}\|\Gamma_{\eps'} D_{\eps}t\|_2-\|t\|_2\Big| \leq \delta$$
with probability at least $1-n^{-\gamma}$ under the assumed bound on $m$ (which then also holds for $n$).}
\endproof

\begin{Remark}
The caveat that $\log(n)\leq d^*(T) \leq n/\log^{5} n$ is there only for the sake of a simpler presentation. In any case, since we are making no attempt of obtaining a result that is accurate at the logarithmic level, that is not a real issue. \red{Indeed, the condition $\log(n)\leq d^*(T)$ can be ensured by replacing $T$ by $T\cup\{e_i, i=1,\ldots,n\}$ - this only leads to an additional logarithmic factor.} If $d^*(T)\geq n/\log^{\red{5}} n$ then replacing $T$ by the Euclidean ball ${\cal R}(T) B_2^n$ comes at most at a logarithmic price, and the latter case can be analyzed directly, by noting that $Q_k(r B_2^n\red{,r}) \sim \rho r \sqrt{n}$.
\end{Remark}

Corollary \ref{cor:ell-1-conc-double circulant} is the final ingredient needed in the proof of Theorem \ref{thm:DoubleCirculantEmbeddingText}.

\vskip0.3cm
\begin{tcolorbox}
\noindent{\bf Proof of Theorem \ref{thm:DoubleCirculantEmbeddingText}.} Recall the three conditions required in Theorem~\ref{thm:mainGeneric}: let $T_\theta$ be a $\theta$-net of $T$ of minimal cardinality and set $\red{k}=\lfloor \delta m/\lambda \rfloor$. One has to show that
\begin{description}
\item{$(1)$} $\sup_{x \in T_\theta} \|Ax\|_{[\red{k}]} \leq \lambda \sqrt{\red{k}}$.
\item{$(2)$} $\sup_{x \in (T-T) \cap \theta B_2^n} \|Ax\|_{[\red{k}]} \leq \delta \sqrt{\red{k}}$.
\item{$(3)$} $\sup_{x \in (T_\theta - T_\theta)} \left|\frac{\kappa}{m}\|Ax\|_1 - \|x\|_2  \ \right| \leq \delta$
\end{description}
for the matrix
$$
A=R_I \Gamma_G D_{\eps^{\prime \prime}} \frac{1}{n^{1/2}} \Gamma_{\eps^\prime} D_\eps = \sqrt{\nrows} B D_\eps.
$$
The proof that the three conditions are satisfied with the wanted probability is an immediate outcome of Theorem~\ref{thm:ugly-est} combined with Theorem \ref{thm:regularity-of-B}---used to establish $(1)$ and $(2)$; and Corollary \ref{cor:ell-1-conc-double circulant}, which implies $(3)$ (for $\kappa=\blue{\sqrt{\pi/2}}$).
\endproof
\end{tcolorbox}

\subsection{Proof of Theorem~\ref{thm:JLell1ell2Intro}}

\green{Let $T^{\prime}$ be the set of normalized differences defined in \eqref{eq:T-prime}. Let $A_I=\sqrt{\tfrac{\pi}{2}}A$ be the rescaled double circulant matrix with index set $I\subset [n]$. Consider independent selectors $\theta_1,\ldots,\theta_n$ satisfying $\bP(\theta_i=1)=1-\bP(\theta_i=0)=\frac{m}{n}$ and let $I_{\theta}=\{i\in [n] \ : \ \theta_i=1\}$ be the set of selected indices. Clearly,
$$\frac{1}{m}\E_{\theta}\|A_{I_{\theta}}z\|_1 = \frac{1}{n}\|A_{[n]}z\|_1$$
for any $z\in \R^n$ and hence
$$\sup_{z\in T'}\left|\frac{1}{m}\|A_{I_{\theta}}z\|_1 - 1\right| \leq \sup_{z\in T'}\left|\frac{1}{m}\|A_{I_{\theta}}z\|_1 - \frac{1}{m}\E_{\theta}\|A_{I_{\theta}}z\|_1\right| + \sup_{z\in T'}\left|\frac{1}{n}\|A_{[n]}z\|_1 - 1\right|.$$
By Corollary~\ref{cor:ell-1-conc-double circulant}, if $n\gtrsim c(\gamma)\epsilon^{-2}\log |T| \log^{6} n$ then
$$\sup_{z\in T'}\left|\frac{1}{n}\|A_{[n]}z\|_1 - 1\right|\leq \epsilon$$ 
with probability at least $1-n^{-\gamma}$.}\par 
\green{To control the first term, set $X_{\theta,z}=\frac{1}{m}\|A_{I_{\theta}}z\|_1 - \frac{1}{m}\E_{\theta}\|A_{I_{\theta}}z\|_1$ and let $\theta'$ be an independent copy of $\theta$. By a symmetrization argument (see (6.3) in \cite{LT91}), 
\begin{equation}
\label{eqn:tailSymm}
\bP_{\theta}\left(\sup_{z\in T'} |X_{\theta,z}|\geq 4\epsilon\right)\leq \bP_{\theta,\theta'}\left(\sup_{z\in T'} |X_{\theta,z}-X_{\theta',z}|\geq 2\epsilon\right) + \sup_{z\in T'} \bP_{\theta}(|X_{\theta,z}|\geq 2\epsilon).\end{equation}
Observe that $X_{\theta,z}-X_{\theta',z}$ has the same distribution as
$$\frac{1}{m}\sum_{i=1}^n \zeta_i(\theta_i-\theta_i')|\langle a_i,z\rangle|,$$
where the $a_i$ are the rows of $A_{[n]}$ and $\zeta$ is a Rademacher vector that is independent of all other random variables. Hence, it is standard to verify that
$$\bP_{\theta,\theta'}\left(\sup_{z\in T'} |X_{\theta,z}-X_{\theta',z}|\geq 2\epsilon\right)\leq 2\bP_{\theta,\zeta}\left(\sup_{z\in T'}\frac{1}{m}\sum_{i=1}^n \zeta_i\theta_i|\langle a_i,z\rangle|\geq \epsilon\right).$$
Observe that the event $\{m/2\leq |I_{\theta}|\leq 3m/2\}$ holds with $\theta$-probability at least $1-e^{-cm}$. On this event, Theorem~\ref{thm:regularity-of-B} shows that with probability at least $1-n^{-\gamma}$ with respect to $G \otimes \eps^\prime \otimes \eps^{\prime \prime}$
$$
B_{I_{\theta}}= \frac{1}{\sqrt{|I_{\theta}|}} R_{I_{\theta}} \Gamma_G D_{\eps^{\prime \prime}} \frac{1}{\sqrt{n}} \Gamma_{\eps^\prime}
$$
is $\rho$-regular for $\rho=c(\gamma)\log^{5/2}(n)/\sqrt{\nrows}$. Since $A_{I_{\theta}} = \sqrt{|I_{\theta}|} B_{I_{\theta}} D_\eps$, we can use Theorem~\ref{thm:Men21} to conclude that with probability at least $1-n^{-\gamma}$ with respect to $G \otimes \eps \otimes \eps^\prime \otimes \eps^{\prime \prime}$,
$$\sup_{z\in T'}\frac{1}{m}\sum_{i=1}^n \theta_i |\langle a_i,z\rangle|^2 = \sup_{z\in T'}\frac{1}{m}\|A_{I_{\theta}}z\|_2^2\leq c,$$
where $c$ is an absolute constant, provided that 
\begin{equation}
\label{eqn:condmNoepsilon}
m\geq c(\gamma) \log|T|\log^5(n).
\end{equation} 
Clearly, \eqref{eqn:condmNoepsilon} is satisfied when $m\geq c(\gamma) \epsilon^{-2}\log|T|$ and $\epsilon\leq \log^{-5/2}(n)$.  In particular, all of the above events hold simultaneously with probability at least $1-e^{-cm}-n^{-\gamma}$ with respect to $\theta\otimes G \otimes \eps \otimes \eps^\prime \otimes \eps^{\prime \prime}$. Conditioned on that event, Hoeffding's inequality implies that 
$$\bP_{\zeta}\left(\frac{1}{m}\sum_{i=1}^n \zeta_i\theta_i|\langle a_i,z\rangle|\geq \epsilon\right)\leq 2e^{-cm\epsilon^2}$$
for any $z\in T^{\prime}$. Recalling that $m\gtrsim \epsilon^{-2}\log|T|$, a union bound yields
$$\bP_{\zeta}\left(\sup_{z\in T'}\frac{1}{m}\sum_{i=1}^n \zeta_i\theta_i|\langle a_i,z\rangle|\geq \epsilon\right)\leq 2e^{-c'm\epsilon^2}.$$
In a similar, but simpler manner one can estimate the second term on the right hand side of \eqref{eqn:tailSymm}, which completes the proof.}  

\subsection{Proof of Theorem \ref{thm:ell-1-conc-double circulant}.}

\blue{As we noted previously, to prove Theorem~\ref{thm:ell-1-conc-double circulant} it remains to show that \eqref{eq:ell-1-conc-A-cond} holds with the parameters specified in \eqref{eq:ell-1-conc-A-condParam}. The random vectors $\eps$, $\eps^\prime$, and $\eps^{\prime \prime}$ each play a different role in the argument, leading to the somewhat cumbersome formulation of Theorem \ref{thm:ell-1-conc-double circulant}.}\par
Firstly, we show that for a typical realization of $\eps^{\prime} \otimes \eps^{\prime \prime}$, the matrix $W D_{\eps^{\prime \prime}}\Gamma_{\eps^\prime}$ is $\red{\rho}$-strongly regular for $\rho$ of the order of $n^{-1/2}\log^{\red{2}} n$. Clearly, that fact has nothing to with the identity of the set $T$, but rather only with the way in which the matrix $W D_{\eps^{\prime \prime}}\Gamma_{\eps^\prime}$ acts on sparse vectors. \blue{Secondly}, given a set $T$, Theorem \ref{thm:Men21} and Theorem \ref{thm:new} imply that (conditioned on the fact that $W D_{\eps^{\prime \prime}}\Gamma_{\eps^\prime}$ is $\red{\rho}$\red{-}strongly regular), with high probability with respect to $\eps$, $W D_{\eps^{\prime \prime}}\Gamma_{\eps^\prime} D_\eps$ `acts well' on $T$. \blue{These two steps are formulated in Theorem~\ref{thm:ell-1-conc-strong-regularity} and Corollary~\ref{cor:ell-1-conc-strong-regularity}.} 
\begin{Theorem} \label{thm:ell-1-conc-strong-regularity}
For $\gamma \geq 1$ there is a constant $c(\gamma)$ that depends only on $\gamma$ such that the following holds. With probability at least $1-n^{-\gamma}$ with respect to $\eps^\prime \otimes \eps^{\prime \prime}$, the matrix $n^{-1/2} W D_{\eps^{\prime \prime}}\Gamma_{\eps^\prime}$ is $\red{\rho}$-strongly regular for
$$
\rho = c(\gamma) \frac{\red{\log^{5/2} n}}{\sqrt{n}}.
$$
\end{Theorem}

We prove Theorem \ref{thm:ell-1-conc-strong-regularity} in the next section. We first show how it implies \eqref{eq:ell-1-conc-A-cond} and, hence, Theorem~\ref{thm:ell-1-conc-double circulant} as well.

\begin{Corollary} \label{cor:ell-1-conc-strong-regularity}
Set $\Upsilon=\rho \sqrt{n}$ and let $T \subset \R^n$. Assume that $\log n \leq d^*(T) \leq c_1 n/\Upsilon^2$ and set $r$ to satisfy that $d^*(T) \sim r \log(en/r)$. Then on the event from Theorem \ref{thm:ell-1-conc-strong-regularity}, with probability at least $1-2\exp(-c_2 u^2 d^*(T))$ with respect to $\eps$, we have that for any $t \in T$,
$$
\|n^{-1/2} W D_{\eps^{\prime \prime}}\Gamma_{\eps^\prime} D_\eps t\|_{[r]} \leq c_3 u^2 \Upsilon \red{\frac{\ell_*(T)}{\sqrt{n}}}
$$
\end{Corollary}
\proof 
On the event from Theorem~\ref{thm:ell-1-conc-strong-regularity}, Theorem~\ref{thm:new} and Remark~\ref{rem:useful-presentation} imply that for $T \subset \R^n$, with probability at least $1-2\exp(-c_0u^2[d^*(T) + r \log(en/r)])$ with respect to $\eps$, for every $t \in T$,
$$
\|n^{-1/2} W D_{\eps^{\prime \prime}}\Gamma_{\eps^\prime} D_\eps t\|_{[r]} \leq c_1 u^2 \max\{\red{Q}(T),{\cal R}^{-1}(T)\red{Q}^2(T)\},
$$
\red{where} 
$$\red{Q(T):=Q_r(T,{\cal R}(T)) = \rho {\cal R}(T)(d^*(T)+r\log(en/r))^{1/2}.}$$
\blue{By setting $r$ to satisfy that $d^*(T) \sim r \log(en/r)$, we have that $\red{Q}(T) \sim \rho {\cal R}(T) \sqrt{d^*(T)}$. Note that $\red{Q}(T) \geq {\cal R}^{-1}(T) \red{Q}^2(T)$ provided that $\red{Q}(T) \leq {\cal R}(T)$. Since $\Upsilon=\rho \sqrt{n}$, this condition is satisfied if $d^*(T) \leq n/\Upsilon^2$. This yields the asserted estimate.} 
\endproof

\subsection{Proof of Theorem \ref{thm:ell-1-conc-strong-regularity}.}

The proof follows two steps:
\vskip0.3cm
\noindent \underline{Step 1: Proof of regularity}.

The first step in the proof is to establish that with high probability with respect to $\eps^\prime \otimes \eps^{\prime \prime}$, the matrix $n^{-1/2} W D_{\eps^{\prime \prime}}\Gamma_{\eps^\prime}$ is regular, i.e., it acts in a norm preserving way on sparse vectors. Observe that for any $t \in \R^n$ and any realization of $\eps^{\prime \prime}$,
$$
\|n^{-1/2} W D_{\eps^{\prime \prime}}\Gamma_{\eps^\prime}\|_2 = \|n^{-1/2}\Gamma_{\eps^\prime}\|_2.
$$
Thus, it suffices to establish the regularity of $n^{-1/2}\Gamma_{\eps^\prime}$. \blue{This follows immediately from} Theorem~\ref{thm:RIP-sparse-simple}, which implies that:
\begin{tcolorbox}
With probability at least $1-2\exp(-\red{c_1}\gamma \log^4 n)$ with respect to $\eps^\prime$, the matrix $n^{-1/2}\Gamma_{\eps^\prime}$ is $\red{\rho}$-regular for
\begin{equation} \label{eq:circulant-regularity-2}
\rho =c_2(\gamma)\frac{\log^{\red{2}} n}{\sqrt{n}}.
\end{equation}
\end{tcolorbox}

\vskip0.3cm
\noindent \underline{Step 2: Proof of strong regularity}.

\blue{Recall that for} any $x \in \R^n$ we have
$\Gamma_{\eps^\prime}x = \Gamma_x \eps^\prime \ \ {\rm  and} \ \ \Gamma_x = \sqrt{n}UD_{Wx}O$, \red{where $U,W$, and $O$ are as in \eqref{eqn:OWUDef}}. Therefore,
$$
n^{-1/2} W D_{\eps^{\prime \prime}} \Gamma_{\eps^\prime} x = W D_{\eps^{\prime \prime}} UD_{Wx} O \eps^\prime.
$$
\blue{Note that this is a random vector of the form $\mathcal{U}D_{\mathcal{W}x} \mathcal{O}\xi$ for} 
$$
{\cal U}=W D_{\eps^{\prime \prime}} U, \ \ {\cal W}=W, \ \ {\cal O}=O \ \ {\rm and} \ \ \xi=\eps^\prime.
$$
\blue{Therefore, to establish strong regularity, we invoke Theorem~\ref{thm:s-normEstimates-sparse}. To that end, observe that the matrix ${\cal W}=W$ is of Hadamard-type; in particular, $d_{\mathcal{W}} = 1$ and trivially, $\sup_{x \in \Sigma_{r,n}} \|\mathcal{W}x\|_2 \leq 2$. To estimate 
$$d_{\mathcal{U}} = \sqrt{n}\max_{1 \leq i,j \leq n} |(W D_{\eps^{\prime \prime}} U)_{ij}|,$$
we use the following fact.}
\begin{Lemma} 
\label{lem:estHadConst}
\blue{There is a constant \red{$c>0$} such that the following holds for any $\gamma \geq 1$ : for any $V_1,V_2\in \C^{n\times n}$, with probability at least $1-n^{-\gamma}$}
$$
\blue{\max_{1 \leq i,j \leq n} |(V_1 D_{\eps^{\prime \prime}} V_2)_{ij}| \leq \red{c}\sqrt{\red{\gamma}\log n}\max_{1 \leq j \leq n}\|V_1^*e_j\|_2 \max_{1 \leq i,j \leq n} |(V_2)_{ij}|.}
$$
\end{Lemma}
The proof of Lemma~\ref{lem:estHadConst} is standard and is presented in Appendix~\ref{app:gaussian}.
\vskip0.3cm
\blue{Since $W$ and $U$ are Hadamard-type matrices, Lemma~\ref{lem:estHadConst} shows that there is an event with probability at least $1-n^{-\gamma}$ (with respect to $\eps^{\prime \prime}$) such that 
$$
d_{{\cal U}} \leq c_1\sqrt{\red{\gamma}\log n}.
$$
On that event, it follows from Theorem~\ref{thm:s-normEstimates-sparse} that} there is an absolute constant $c_2$, such that with probability at least $1-2\exp(-c_2u)$ with respect to $\eps^\prime$,
\begin{equation} \label{eq:ell-1-conc-circ-1-on-sparse-0}
\sup_{x\in \Sigma_{r,n}}\|W D_{\eps^{\prime \prime}} \Gamma_{\eps^\prime} x \|_{[r]} \leq \red{c_3} \sqrt{\frac{r}{n}} \cdot \sqrt{\log n}\left(\log(n)\log(r) +  \sqrt{u}\right).
\end{equation}

In particular, setting $u \sim \blue{\gamma} \log^4 n$ and taking the union bound over $1 \leq r \leq n$, we have that with probability at least $1-2\exp(-\red{c_2}\gamma \log^4 n)$ with respect to $\eps^\prime$,

\begin{equation} \label{eq:ell-1-conc-circ-1-on-sparse-1}
\sup_{x\in \Sigma_{r,n}}\|n^{-1/2} W D_{\eps^{\prime \prime}} \Gamma_{\eps^\prime} x \|_{[r]} \leq c_4(\gamma) \sqrt{\frac{r}{n}} \log^{5/2} n  \ \ \ {\rm for \ every \ } 1 \leq r \leq n.
\end{equation}
\blue{In particular:}
\begin{tcolorbox}
On an event with probability at least $1-n^{-\gamma}-\exp(-c_5\gamma \log^4 n)$ with respect to $\eps^\prime \otimes \eps^{\prime \prime}$, $n^{-1/2} W D_{\eps^{\prime \prime}} \Gamma_{\eps^\prime}$ is $\red{\rho}$-strongly regular for
$$
\rho = c_4(\gamma) \frac{\log^{\red{5/2}} n}{\sqrt{n}}.
$$
\end{tcolorbox}
That completes the proof of Theorem \ref{thm:ell-1-conc-strong-regularity}.
\endproof

\section*{Acknowledgements}

S.D.\ was supported by the Deutsche Forschungsgemeinschaft (DFG, German Research Foundation) under SPP 1798 (COSIP - Compressed Sensing in Information Processing) through project CoCoMiMo.  A.S.\ acknowledges support by the Fonds de la Recherche Scientifique - FNRS under Grant n$^{\circ}$ T.0136.20 (Learn2Sense).

\appendix

\section{Proof of Lemma~\ref{lem:estHadConst}} \label{app:gaussian}

\blue{For any $1\leq i,j\leq n$,
$$(V_1 D_{\eps^{\prime \prime}} V_2)_{ij} = \langle V_1 D_{\eps^{\prime \prime}} V_2e_i,e_j\rangle = \langle \eps^{\prime \prime}, D_{V_2e_i}^*V_1^*e_j\rangle.$$
Moreover, 
$$\|D_{V_2e_i}^*V_1^*e_j\|_2\leq \|D_{V_2e_i}\|_{2\to 2} \|V_1^*e_j\|_2 = \|V_2e_i\|_{\infty}  \|V_1^*e_j\|_2 \leq \max_{1\leq j\leq n}  \|V_1^*e_j\|_2 \max_{1 \leq i,j \leq n}|(V_2)_{ij}|.$$ 
By Hoeffding's inequality, for any $u>0$,
$$\bP\Big(|\langle \eps^{\prime \prime}, D_{V_2e_i}^*V_1^*e_j\rangle|\geq u \max_{1\leq j\leq n} \|V_1^*e_j\|_2 \max_{1 \leq i,j \leq n}|(V_2)_{ij}|\Big) \leq 2\exp(-u^2/2).$$
The result now follows by taking a union bound over all $i$ and $j$.}  

\section{Proof of Theorem \ref{thm:regularity-of-B}} \label{sec:strong-regularity-double-circ-proof}

The goal is to show that with probability at least $1-n^{-\gamma}$ with respect to $G \otimes \eps^{\prime} \otimes \eps^{\prime \prime}$, the matrix $B=\nrows^{-1/2} R_I \Gamma_G D_{\eps^{\prime \prime}} n^{-1/2}\Gamma_{\eps^\prime}$
is $\red{\rho}$-strongly regular for $\rho \sim \red{c}(\gamma) \nrows^{-1/2}\log^{\red{5/2}} n$.

Recall that
$\Psi=n^{-1/2} D_{\eps^{\prime \prime}} \Gamma_{\eps^\prime}$, and thus, for every $x \in \R^n$,
$$
Bx=\frac{1}{\sqrt{\nrows}} R_I \Gamma_G \Psi x = \frac{1}{\sqrt{\nrows}} R_I \Gamma_{\Psi x} G.
$$

\subsection{\blue{Regularity of $B$}}

\blue{As we have noted in \eqref{eq:RIP-B-2},} to establish regularity it suffices to estimate
\begin{align} \label{eq:RIP-B-2-app}
& \sup_{x \in \Sigma_{r,n}} \left| \ \|Bx\|_2^2 - \|x\|_2^2 \ \right| \nonumber
\\
\leq & \sup_{x \in \Sigma_{r,n}}  \left| \ \left\|\frac{1}{\sqrt{\nrows}} R_I \Gamma_{\Psi x} G \right\|_2^2 - \E_G \left\|\frac{1}{\sqrt{\nrows}} R_I \Gamma_{\Psi x} G \right\|_2^2 \ \right| + \sup_{x \in \Sigma_{r,n}} \left| \left\|\frac{1}{\sqrt{n}} \Gamma_x \eps^\prime \right\|_2^2 - \blue{\|x\|_2^2} \ \right|.
\end{align}
As it happens, the proof that $n^{-1/2}\Gamma_{\eps^\prime}$ is $\red{\rho}$-regular for $\rho \sim \red{c}(\gamma) \frac{\log^{\red{2}} n}{\sqrt{n}}$ is standard and was used previously in this presentation (see \eqref{eq:circulant-regularity-2}). Hence, all that remains is to estimate
$$
\sup_{x \in \Sigma_{r,n}} \left| \ \left\|\frac{1}{\sqrt{\nrows}} R_I \Gamma_{\Psi x} G \right\|_2^2 - \E_G \left\|\frac{1}{\sqrt{\nrows}} R_I \Gamma_{\Psi x} G \right\|_2^2 \ \right|
$$
\blue{for $1 \leq r \leq \nrows$}. To that end we invoke Theorem~\ref{thm:KMR14} for the class of matrices
$$
{\cal A}_r = \left\{ R_I \Gamma_{\Psi x} : x \in \Sigma_{r,n} \right\}
$$
in the case $\xi=G$. To estimate the quantities in Theorem~\ref{thm:KMR14} we follow an almost identical argument to the one used in \blue{\cite{KMR14}}. We will therefore only outline it here. Let $x,y \in \R^n$. Then
$$
\|R_I \Gamma_{\Psi x} - R_I \Gamma_{\Psi y}\|_{2 \to 2} \leq \|\Gamma_{\Psi(x-y)}\|_{2 \to 2} = \sqrt{n}\|\Psi(x-y)\|_{\infty},
$$
and \blue{in particular}
$$
\sup_{x \in \Sigma_{r,n}} \|R_I \Gamma_{\Psi x}\|_{2 \to 2} \leq \sup_{x \in \Sigma_{r,n}} \sqrt{n}\|\Psi x\|_{\infty}.
$$
Set $\blue{\|x\|:=\sqrt{n}\|\Psi x\|_\infty}$. Estimating the $\gamma_2$-functional by an entropy integral (see e.g.\ \cite{Tal14}), we find that for an absolute constant $c_0$,
\begin{equation}
\label{eqn:entropyIntEst}
\gamma_2({\cal A}_r,\| \cdot \|_{2 \to 2}) \leq \gamma_2(\Sigma_{r,n},\blue{\|\cdot\|}) \leq c_0\int_0^\infty \log^{1/2} {\cal N}(\Sigma_{r,n}, \blue{\|\cdot\|,u}) \ du.
\end{equation}
\blue{To estimate the right hand side, we use two entropic estimates. The first is based on Maurey's Lemma and is essentially due to Carl \cite{Car85} (see also \cite{KMR14} for a proof).}
\begin{Lemma} \label{lemma:covering-app}
There exists an absolute constant $c$ such that the following holds. \blue{Let $\| \cdot \|$ be a norm on $\R^n$.} Let $U \subset \R^n$ be a finite set and assume that for every $1 \leq k \leq |U|$ and every subset $\{u_1,...,u_k\} \subset U$ of cardinality $k$, $\E_\eps \|\sum_{i=1}^k \eps_i u_i \| \leq \alpha \sqrt{k}$. Then for every $t>0$,
$$
\log {\cal N}({\rm conv}(U), \|\cdot\|, t) \leq c\left(\frac{\alpha}{t}\right)^{\blue{2}} \log |U|.
$$
\end{Lemma}
In our case, $\Sigma_{r,n} \subset {\rm conv}(U)$, where $U = \{\pm 2 \sqrt{r}e_i : 1 \leq i \leq n\}$, and $\|x\| = \sqrt{n}\|\Psi x\|_\infty$. Note that for any $x_1,...,x_k \in \R^n$ 
$$
\E \left\|\sum_{i=1}^k \eps_i x_i \right\|_{\infty} \leq c_0 \sqrt{\log n} \left(\sum_{i=1}^k \|x_i\|_\infty^2 \right)^{1/2}.
$$
Hence, if $J \subset \{1,...,n\}$ and $|J|=k$, we have that
\begin{align*}
\E_\eps \left\| \sum_{i\in J} \eps_i e_i \right\| & = \E_\eps \left\| \sum_{i\in J} \eps_i \sqrt{n}\Psi e_i \right\|_{\infty} \leq c_0 \sqrt{\log n} \left(\sum_{i\in J} \|\sqrt{n}\Psi e_i\|_{\infty}^2 \right)^{1/2}
\\
& \leq c_0 \sqrt{\log n} \max_{1 \leq i,j \leq n} \sqrt{n}|\Psi_{ij}| \cdot \sqrt{k}.
\end{align*}
\blue{Now, condition on the event on which 
\begin{equation}
\label{eqn:hadEstPsi}
\max_{1 \leq i,j \leq n} \sqrt{n} |\Psi_{ij}| \leq c_1\sqrt{\red{\gamma}\log n}.
\end{equation}
Since $\Psi=n^{-1/2} D_{\eps^{\prime \prime}} \Gamma_{\eps^\prime}$, Lemma~\ref{lem:estHadConst} (with $V_1=\operatorname{Id}_{\red{n}}$ and $V_2=\frac{1}{\sqrt{n}}\Gamma_{\eps^{\prime}}$) shows that this event holds with probability at least $\red{1-n^{-\gamma}}$ with respect to $\eps^{\prime \prime}$. On that event, Lemma~\ref{lemma:covering-app} with 
$$
\alpha=c_0 \sqrt{r} \sqrt{\log n} \max_{1 \leq i,j \leq n} \sqrt{n}|\Psi_{ij}|\leq \red{c_2(\gamma)} \sqrt{r} \log n
$$ 
implies that} for every $t >0$
$$
\log {\cal N}(\Sigma_{r,n}, \blue{\|\cdot\|, t}) \leq \red{c_3}(\gamma) \frac{r}{t^2} \log^{\blue{3}} n.
$$

The second entropic estimate \blue{we require is volumetric:} $\Sigma_{r,n}$ is the union of the Euclidean unit balls $B_2^J$ that are supported on sets $J \subset \{1,...,n\}$ where $|J|=r$. If $x,y \in B_2^J$ then
\begin{align*}
\|x-y\|= & \sqrt{n}\|\Psi(x-y)\|_{\infty} \leq \max_{1 \leq i \leq n} |\sqrt{n} \inr{\Psi^*e_i,x-y}| \leq \max_{1 \leq i,j \leq n} \sqrt{n}|\Psi_{ij}| \|x-y\|_1
\\
\leq & \max_{1 \leq i,j \leq n} \sqrt{n}|\Psi_{ij}| \sqrt{r} \|x-y\|_2 \leq \red{c_1}(\gamma)\sqrt{\log n} \sqrt{r} \|x-y\|_2,
\end{align*}
\blue{where we again used \eqref{eqn:hadEstPsi}.} Hence, for any $t >0$
\begin{equation*}
\log {\cal N}(\Sigma_{r,n},\blue{\|\cdot\|,t}) \blue{=} \log {\cal N}\left(\cup_{|J|=r} B_2^J, \blue{\|\cdot\|,t}\right) \leq \max_{|J|=r} \log {\cal N}(B_2^J, \blue{\|\cdot\|_2,t^\prime} )+ r \log(en/r),
\end{equation*}
where
$$
t^\prime = \frac{t}{\red{c_1}(\gamma)\sqrt{\log n} \sqrt{r}}.
$$
Now the entropy estimate follows from a standard volumetric argument.

Conditioned on \blue{the event \eqref{eqn:hadEstPsi}} and using these two entropic estimates for large and small $u$, respectively, in the entropy integral in \eqref{eqn:entropyIntEst}, it follows from a standard computation that for any $1 \leq r \leq n$,
$$
\gamma_2({\cal A}_r,\| \cdot \|_{2 \to 2}) \leq c(\gamma) \sqrt{r} \log^{\blue{5/2}} n. 
$$
\blue{A similar argument shows that $d_{2 \to 2} ({\cal A}_r) \leq c(\gamma) \sqrt{r} \sqrt{\log n}$}.\par
\blue{Finally, let us estimate 
$$d_{HS}({\cal A}_r) = \sup_{x \in \Sigma_{r,n}} \|R_I \Gamma_{\Psi x}\|_{HS} \leq \sqrt{m}\sup_{x \in \Sigma_{r,n}} \|\Psi x\|_2.$$
Observe that
\begin{equation*}
\sup_{x \in \Sigma_{r,n}} \|\Psi x\|_2 = \frac{1}{\sqrt{n}} \sup_{x \in \Sigma_{r,n}} \|\Gamma_{\eps^\prime}x\|_2.
\end{equation*}
We have that $\E\|\Gamma_x \eps^\prime\|_2^2=n\|x\|_2^2 \leq n$, and if
$$
a:=\sup_{x \in \Sigma_{r,n}} \left| \ \|\Gamma_x \eps^\prime\|_2^2 -\E\|\Gamma_x \eps^\prime\|_2^2 \ \right|,
$$
then
$$
\frac{1}{\sqrt{n}} \sup_{x \in \Sigma_{r,n}} \|\Gamma_{\eps^\prime}x\|_2 \leq \frac{1}{\sqrt{n}}\left( a + n \right)^{1/2}.
$$
Using \red{Theorem~\ref{thm:RIP-sparse-simple}} it follows that with probability at least $1-2\exp(-\red{c_0\gamma\log^4(n)})$ with respect to $\eps^\prime$, for all $1 \leq r  \leq \nrows$
$$
\sup_{x \in \Sigma_{r,n}} \left| \ \|\Gamma_x \eps^\prime\|_2^2 -\E\|\Gamma_x \eps^\prime\|_2^2 \ \right| \leq n,
$$
\red{provided that $n\geq c_1\gamma m\log^4(n)$,} implying that
\begin{equation}
\label{eqn:Psix2est}
\sup_{x \in \Sigma_{r,n}} \|\Psi x\|_2 = \frac{1}{\sqrt{n}} \sup_{x \in \Sigma_{r,n}} \|\Gamma_{\eps^\prime}x\|_2 \leq 2.
\end{equation}
Therefore, conditioned on the above event it follows that}
$$
d_{HS}({\cal A}_r) \leq \sqrt{\nrows}.
$$
\blue{Combining these estimates with Theorem~\ref{thm:KMR14} \red{and a union bound over all $1 \leq r  \leq \nrows$}, we have that} with probability at least $1-2\exp(-c_3 \gamma \log^4 n)$ with respect to $G$, the matrix $B$ is $\red{\rho}$-regular for $\rho=c(\gamma)\frac{\log^{\red{5/2}} n}{\sqrt{\nrows}}$.

\subsection{Strong regularity of $B$}
To complete the proof of Theorem \ref{thm:regularity-of-B} one has to show that for every $1 \leq r \leq \nrows$,
$$
\sup_{x \in \Sigma_{r,n}} \left\|\frac{1}{\sqrt{\nrows}} R_I \Gamma_{\Psi x} G\right\|_{[r]} \leq \rho \sqrt{r}.
$$
To that end, it suffices to prove that
$$
\sup_{x \in \Sigma_{r,n}} \left\|\frac{1}{\sqrt{\nrows}} \Gamma_{\Psi x} G\right\|_{[r]} = \sqrt{\frac{n}{\nrows}} \sup_{x \in \Sigma_{r,n}} \|UD_{W\Psi x}OG\|_{[r]} \leq \rho \sqrt{r}.
$$
\blue{Thus, the proof of Theorem \ref{thm:regularity-of-B} is completed once the following lemma is established.}
\begin{Lemma} \label{lemma:strong-regularity-circ}
For $\gamma>1$ there exist constants $c_0$ and $c_1$ \red{depending only (polynomially) on $\gamma$} such that the following holds. Let $n \geq c_0\nrows \log^{\red{4}} n$. With probability at least $1-n^{-\gamma}$ with respect to $G \otimes \eps^{\prime} \otimes \eps^{\prime \prime}$, for every $1 \leq r \leq \nrows$,
$$
\sup_{x \in \Sigma_{r,n}} \|UD_{W\Psi x}OG\|_{[r]} \leq c_1 \sqrt{\frac{r}{n}}\log^{5/2} n.
$$
\end{Lemma}
\proof \blue{Using the notation of Theorem~\ref{thm:s-normEstimates-sparse}, 
$$UD_{W\Psi x}OG=\mathcal{U}D_{\mathcal{W}x} \mathcal{O}\xi,$$
where 
$$
{\cal U}=U, \ \ {\cal W}=W\Psi, \ \ {\cal O}=O \ \ {\rm and} \ \ \xi=G.
$$
Just as in \eqref{eqn:Psix2est}, with probability at least $1-2\exp(-\red{c_0\gamma \log^4(n)})$ with respect to $\eps^\prime$, for all $1 \leq r  \leq \nrows$,
\begin{equation*}
\sup_{x \in \Sigma_{r,n}} \|{\cal W} x\|_2 = \sup_{x \in \Sigma_{r,n}} \|W\Psi x\|_2 = \frac{1}{\sqrt{n}} \sup_{x \in \Sigma_{r,n}} \|\Gamma_{\eps^\prime}x\|_2 \leq 2.
\end{equation*}
Moreover, by Lemma~\ref{lem:estHadConst} (with $V_1=W$ and $V_2=\frac{1}{\sqrt{n}}\Gamma_{\eps^{\prime}}$) 
\begin{equation*}
d_{\mathcal{W}}=\max_{1 \leq i,j \leq n} \sqrt{n} |(W\Psi)_{ij}| \leq c_1 \sqrt{\red{\gamma}\log n}
\end{equation*}
with probability at least $1-\red{n^{-\gamma}}$ with respect to $\eps^{\prime \prime}$. The result now follows by applying Theorem~\ref{thm:s-normEstimates-sparse} conditioned on those two events.} 
\endproof

\end{document}